\documentclass[12pt]{amsart}

\usepackage{ xy, graphicx, color, hyperref}
 \usepackage{amsmath, amssymb}
\usepackage{amsfonts,  hyperref}

 \usepackage{verbatim}

\usepackage{tikz}
\usetikzlibrary{positioning}
\usepackage{amsfonts}

\xyoption{all}

\title[Spinoriality]{Spinoriality of Orthogonal Representations of Reductive Groups}
\author{Rohit Joshi}
\author{Steven Spallone}

\newtheorem{theorem}{Theorem}
\newtheorem{corollary}{Corollary}
\newtheorem{lemma}{Lemma}

\newtheorem{prop}{Proposition}
\newtheorem{defn}{Definition}

\theoremstyle{definition}
\newtheorem{remark}{Remark}
\newtheorem{example}{Example}

\newcommand{\nc}{\newcommand}

\nc{\thm}{\theorem}
\nc{\cor}{\corollary}
\nc{\mc}{\mathcal}
\nc{\mb}{\mathbb}
\nc{\mf}{\mathfrak}
\nc{\ul}{\underline}
\nc{\ol}{\overline}
\nc{\N}{\mb N}
\nc{\R}{\mb R}
\nc{\Z}{\mb Z}
\nc{\Q}{\mb Q}
\nc{\gm}{\gamma}
\nc{\vt}{\vartheta}

\nc{\dnu}{\frac{\partial}{\partial \nu}}

\nc{\Dnu}[1]{\frac{\partial^{#1}}{\partial \nu^{#1}}}

\nc{\dmo}{\DeclareMathOperator}

\dmo{\Ker}{Ker} \dmo{\val}{val} \dmo{\ord}{ord}

 \dmo{\pr}{pr}
\dmo{\I}{I}
\dmo{\II}{II}
\dmo{\odd}{odd}
\dmo{\sgn}{sgn}
\dmo{\id}{id}
\nc{\beq}{\begin{equation*}}
\nc{\eeq}{\end{equation*}}
\nc{\half}{\frac{1}{2}}
\dmo{\Mod}{mod}
\dmo{\dyn}{dyn}
\dmo{\simp}{sc}
\dmo{\im}{im}
\dmo{\Ad}{Ad}
\dmo{\Gal}{Gal}

\dmo{\der}{der}

\dmo{\core}{core}
\dmo{\res}{res}
\dmo{\lin}{lin}
 \dmo{\Spin}{Spin}
\dmo{\Sp}{Sp}
\dmo{\SL}{SL}
\dmo{\GL}{GL}
\dmo{\SO}{SO}
\dmo{\PGL}{PGL}
\dmo{\PSO}{PSO}
\dmo{\Pin}{Pin}
\dmo{\GSp}{GSp}
\dmo{\sd}{sd}
\dmo{\orth}{orth}
\dmo{\Span}{Span}

\dmo{\scc}{sc}
\dmo{\ad}{ad}

\nc{\la}{\lambda}

 \nc{\lip}{\langle}
 \nc{\rip}{\rangle}

\dmo{\Perm}{Perm}
\dmo{\Res}{Res}
\dmo{\Ind}{Ind}
\dmo{\tr}{tr}
\dmo{\Sym}{Sym}
\dmo{\reg}{reg}
\dmo{\ch}{ch}
\dmo{\Hom}{Hom}
\dmo{\diag}{diag}

\nc{\eps}{\varepsilon}

 \newcommand{\mat}[4]{
    \begin{pmatrix}
      #1 & #2 \\
      #3 & #4
    \end{pmatrix}
}

\address{Bhaskaracharya Pratishthana 
	56/14, Erandavane, Damle Path,
Off Law College Road, Pune - 411 004,Maharashtra,India}
\email{rohitsj@students.iiserpune.ac.in}
\address{Indian Institute of Science Education and Research, Pune-411021,Maharashtra,India}
\email{sspallone@gmail.com}
\date{\today}

\begin{document}
\maketitle
 
\begin{abstract}  
 Let $G$ be a connected reductive group over a field $F$ of characteristic $0$, and $\varphi: G \to \SO(V)$ an orthogonal representation over $F$.  We give criteria to determine when  $\varphi$ lifts to the double cover $\Spin(V)$.  
 \end{abstract}

\tableofcontents

\section{Introduction}
 
 Let $G$ be a connected reductive group over a field $F$ of characteristic $0$.  Let $(\varphi,V)$ be a representation of $G$, which in this paper always means a finite-dimensional $F$-representation of $G$.  Suppose that $V$ is orthogonal, i.e., carries a symmetric nondegenerate bilinear form preserved by $\varphi$.  Thus $\varphi$ is a morphism from $G$ to $\SO(V)$.  Write $\rho: \Spin(V) \to \SO(V)$ for the usual isogeny (\cite{Veldkamp}).  Following \cite{Bou.Lie.7-9}, we say that $\varphi$ is {\it spinorial} when it lifts to $\Spin(V)$, i.e., provided there exists a morphism $\hat \varphi: G \to \Spin(V)$  so that $\varphi=\rho \circ \hat\varphi$.  We call $\varphi$ {\it aspinorial} otherwise.   
 
By an argument in Section \ref{reduction.closed}, we may assume that $F$ is algebraically closed, which we do for the rest of this introduction. Let $T$ be a maximal torus of $G$.  Write $\pi_1(G)$ for the fundamental group of $G$ (the  cocharacter group of $T$ modulo the subgroup $Q(T)$ generated by coroots), and $T_V$ for a maximal torus of $\SO(V)$ containing   $\varphi(T)$.  Then $\varphi$ induces a homomorphism $\varphi_*: \pi_1(G) \to \pi_1(\SO(V)) \cong \Z/ 2 \Z$, and $\varphi$ is spinorial iff $\varphi_*$ is trivial.  If we take a set  of cocharacters $\ul \nu=\{ \nu_1, \ldots, \nu_r\}$ whose images generate $\pi_1(G)$, then $\varphi$ is spinorial iff each cocharacter $\varphi_*\nu_i$ of $T_V$ lifts to $\Spin(V)$.   (See Section  \ref{s3}.)

Write $\mf g$ for the Lie algebra of $G$, and $X^*(T)$ for the character group of $T$.  Suppose $(\varphi,V)$ is an orthogonal representation of $G$.  Write $C$ for the  Casimir element  associated to the Killing form. Given a cocharacter $\nu$ of $T$, put
\beq
|\nu|^2=\sum_{\alpha \in R} \lip \alpha,\nu \rip^2 \in 2 \Z.
\eeq
 We introduce the integer
\beq
 p(\ul \nu)=\half \gcd \left( |\nu_1|^2, \ldots, |\nu_r|^2 \right).
\eeq

\begin{thm} \label{mthm} Suppose that $\mf g$ is simple and let $\varphi$ be an orthogonal representation of $G$. Then $\varphi$ is spinorial iff the integer
 \begin{equation} \label{theorem}
 p(\ul \nu) \cdot \frac{\tr(C,V)}{ \dim \mf g}
 \end{equation}
  is even.
 \end{thm}

\bigskip

Alternatively, this can be  reformulated in terms of the Dynkin index `$\dyn(\varphi)$' of $\varphi$ and the dual Coxeter number $\check h$ of $\mf g$. (We recall these integers in Section \ref{s7}.)
   
\begin{cor} \label{berg} Suppose $\mf g$ is simple and let $\varphi$ be an orthogonal representation of $G$.   Then $\varphi$ is spinorial iff the integer
\beq
 p(\ul \nu) \cdot  \frac{\dyn(\varphi)}{2\check h}
\eeq
is even.
\end{cor}

\bigskip
 
If $\la \in X^*(T)$ is dominant, write   $\varphi_\la$ for the irreducible representation with  highest weight $\la$.    As $\la$ varies, we may regard \eqref{theorem} as a integer-valued polynomial in $\la$. We show that the ``spinorial weights" form a periodic subset of the highest weight lattice.  To be more precise, let $X_{\orth}^+ \subset X^*(T)$ be the set of highest weights of irreducible orthogonal representations.  
 
 \begin{thm} \label{period.intro}  There is a $k \in \N$ so that for all $\la_0,\la \in X_{\orth}^+$, the representation $\varphi_{\la_0}$ is spinorial iff $\varphi_{\la_0+2^k \la}$ is spinorial.
\end{thm}
 
\bigskip
 
 For any representation $\varphi$, one can form an orthogonal representation $S(\varphi)=\varphi \oplus \varphi^\vee$.  When $G$ is semisimple, $S(\varphi)$ is always spinorial.  For the reductive case  we have:

\begin{thm}  \label{s.phi.thm} $S(\varphi_\la)$ is spinorial iff the integers
\beq
\lip \la, \nu^z \rip \cdot \dim V_{\la}
\eeq
are even for all   $\nu \in \ul \nu$.
\end{thm}

\bigskip

In this formula, $\nu^z$ is the $\mf z$-component of $\nu$ corresponding to the decomposition $\mf g=\mf g' \oplus \mf z$, where $\mf z$ is the center of $\mf g$ and $\mf g'$ is the derived algebra of $\mf g$.   
  
This paper is organized as follows.   Section \ref{s2} establishes general notation and  Section \ref{tori.spin.section} sets up preliminaries for the spin groups.     In Section \ref{s3} we give a criterion for spinoriality in terms of the weights of $\varphi$.   This approach is along the lines of \cite{Prasad.Ramakrishnan} and   \cite{Bou.Lie.7-9}.   

We advance the theory in  Section \ref{s4} by employing an algebraic trick involving palindromic Laurent polynomials; this gives a lifting condition in terms of the integers
\beq
q_\varphi(\nu)= \half \cdot \frac{d^2}{dt^2} \Theta_\varphi(\nu(t))|_{t=1}
\eeq
for  $\nu \in \ul \nu$.  Here $\Theta_\varphi$ denotes the character of $\varphi$.

 In Section \ref{s5}   we compute $q_\varphi(\nu)$ for $\varphi$ irreducible, essentially by taking two derivatives of Weyl's Character Formula.  As a corollary we show that every nonabelian reductive group has a nontrivial spinorial irreducible representation.   Section \ref{reducible.rep.section}  works out the case of reducible orthogonal representations, in particular we prove Theorems \ref{mthm} and  \ref{s.phi.thm}.  
     In Section \ref{s7} we explain the connection with the Dynkin index.    Spinoriality for tensor products is understood in Section \ref{tensor.section}.  
 
 The next four sections apply our theory to groups $G$  with $\mf g$ simple. Our goal is to answer the question: For which such $G$ is every orthogonal representation spinorial?
 Section  \ref{a.section} covers quotients of $\SL_n$, Section \ref{c.section} covers type $C_n$, Section \ref{d.section} covers type $D_n$, and Section \ref{summary.section} presents the final answer to the question.
 
  In Section \ref{s10} we prove Theorem \ref{period.intro}, the periodicity of the spinorial weights.  Finally, in Section \ref{reduction.closed} we reduce to the case of $F$ algebraically closed.

 \bigskip

{\bf Acknowledgements:} We would like to thank Dipendra Prasad for his interest and useful conversations, and Raghuram and Jeff Adler for helpful suggestions.  This paper comes out of the first author's Ph.D. thesis \cite{Rohit.thesis} at IISER Pune, during which he was supported by an Institute Fellowship. Afterwards he was supported by a fellowship from Bhaskaracharya Pratishthan. We would like to thank the  referee for useful  comments and suggetions on an earlier version of this manuscript.

 \section{Preliminaries} \label{s2} 
 
 \subsection{Notation}
Throughout this paper $G$ is a connected reductive algebraic group over $F$ with Lie algebra $\mf g$.  Until the final section, $F$ is algebraically closed.
Write $\mf g'$ for the derived algebra of $\mf g$.  We write $T$ for a maximal torus of $G$, with  Lie algebra $\mf t$ and Weyl group $W$.  Put $\mf t'=\mf t \cap \mf g'$.  Let $\sgn: W \to \{ \pm 1 \}$ be the usual sign character of $W$. As in \cite{Springer}, let $(X^*,R,X_*,R^{\vee})$ be the root datum associated to $G$, 
 
The groups $X^*=X^*(T)=\Hom(T,\mb G_m)$ and $X_*=X_*(T)=\Hom(\mb G_m,T)$ are the character and cocharacter lattices of $T$.    
One has injections $X^* \hookrightarrow \mf t^*$ and $X_* \hookrightarrow \mf t$ given by differentiation for the former, and $\nu \mapsto d\nu(1)$ for the latter.  We will often   identify $X^*, R, X_*$, and $R^\vee$ with their images under these injections.    
Let $Q(T) \subseteq X_*(T)$ be the group generated by the coroots of $T$ in $G$.  
Write $R^+$ for a set of positive roots of $T$ in $G$, and $\delta \in \mf t^*$ for the half-sum of these positive roots.     Let $w_0 \in W$ denote the longest Weyl group element.

For $\la,\la' \in X^*(T)$, we write   $\la' \prec \la$ when $\la-\la'$ is a nonnegative combination of positive roots.

  In this paper all representations $V$ of $G$ are finite-dimensional $F$-representations, equivalently morphisms $\varphi: G \to \GL(V)$ of algebraic groups.  
For $\mu \in X^*(T)$, write $V^\mu$ for the $\mu$-eigenspace of $V$, and put $m_\varphi(\mu)=\dim V^{\mu}$, the multiplicity of $\mu$ as a weight of $V$.  

If $H$ is an algebraic group, write $H^\circ$ for the connected component of the identity.
We frequently write $\diag(t_1, t_2, \ldots, t_n)$ for the $n \times n$ matrix with the given elements as entries.

\subsection{Pairings} \label{pairings.section}

 Write $\lip \: , \:\rip_T: X^*(T) \times X_*(T) \to \Z$ for the pairing 
 \beq
 \lip \mu,\nu \rip_T=n \Leftrightarrow \mu(\nu(t))=t^n  
 \eeq 
 for $t \in F^\times$, and $\lip \:,\:\rip_{\mf t}: \mf t^* \times \mf t$ for the natural pairing.  Note that for $\mu \in X^*(T)$ and $\nu \in X_*(T)$, we have
\beq
\lip d\mu,d \nu(1) \rip_{\mf t}=\lip \mu,\nu \rip_T.
\eeq
 So we may  drop the subscripts and simply write `$\lip \mu,\nu \rip$'.

Write $( \:,\:)$ for the Killing form of $\mf g$ restricted to $\mf t$; it may be computed by
 \beq
 (x,y)=\sum_{\alpha \in R} \alpha(x)\alpha(y),
 \eeq 
for $x,y \in \mf t$.  Also set $|x|^2=(x,x)$.  In particular, for $\nu \in X_*(T)$ we have $|\nu|^2=\sum_{\alpha \in R} \lip \alpha,\nu \rip^2$.  The Killing form restricted further to $\mf t'$ induces an isomorphism $\sigma: (\mf t')^* \cong  \mf t'$.  We use the same notation `$( \:, \:)$' to denote the inverse form on $(\mf t')^*$ defined for $\mu_1,\mu_2 \in \mf t'$ by
 \beq
 (\mu_1,\mu_2)=(\sigma(\mu_1),\sigma(\mu_2)).
 \eeq
 In  \cite{Bou.Lie.4-6} this form on $(\mf t')^*$ is called the ``canonical bilinear form" $\Phi_R$.
  Write $|y|^2=(y,y)$ for  $y \in (\mf t')^*$.  
  
  Let $\pi_1(G)=X_*(T)/Q(T)$. As in the introduction, fix a set $\ul \nu=\{\nu_1, \ldots, \nu_r\}$ of cocharacters whose images generate $\pi_1(G)$, 
and put
\beq
 p(\ul \nu)=\half \gcd \left( |\nu_1|^2, \ldots, |\nu_r|^2 \right).
\eeq

 Often $\ul \nu$ will be a singleton $\{ \nu_0 \}$, in which case we may simply write 
 \beq
 p(\nu_0)=p(\{\nu_0\})=\half |\nu_0|^2.
 \eeq

\subsection{Orthogonal Representations}

Let $X^*(T)^+$ be the set of dominant characters, i.e., the $\lambda \in X^*(T)$ so that $\lip \lambda, \alpha^\vee \rip \geq 0$ for all $\alpha \in R^+$.

 Put 
\beq
X_{\sd}=\{ \la \in X^*(T) \mid w_0 \la=-\la \}
\eeq
 and 
\beq
X_{\orth}=\{ \la \in X_{\sd} \mid \lip \la,2\delta^\vee \rip \text{ is even} \},
\eeq
and use the superscript `$+$' to denote the dominant members of these sets.
  According to \cite{Bou.Lie.7-9}, $X_{\sd}^+$ is the set of highest weights of irreducible self-dual representations, and $X_{\orth}^+$ is the set of highest weights of irreducible orthogonal representations.

For $\la \in X^*(T)^+$, the quantity 
\beq
|\lambda+\delta|^2-|\delta|^2=(\lambda,\lambda+2\delta)
\eeq
 is equal to  $\chi_\la(C)$, the value of the central character of the irreducible representation $\varphi_\la$ at the Casimir element $C$. (See \cite{KRV}.)

\subsection{Tori of Spin Groups} \label{tori.spin.section}

In this section we recall material about the tori of spin groups. Our reference is Section 6.3 of  \cite{Goodman.Wallach}.

For the even-dimensional case, let $V$ be a vector space with basis $(e_1, \ldots, e_n, e_{-n},\ldots, e_{-1})$.
For the odd-dimensional case, use the basis $(e_1, \ldots, e_n, e_0, e_{-n},\ldots, e_{-1})$. In either case, give $V$ the symmetric bilinear form $( \:,\:)$ so that $(e_i,e_{-i})=1$ and $(e_i,e_j)=0$ for $j \neq -i$.

Let $C(V)$ be the corresponding Clifford algebra, i.e., the quotient of the tensor algebra of $V$ by the relation 
\beq
v \otimes w+w \otimes v=(v,w).
\eeq

Let $\Pin(V)$ denote the subgroup of the invertible elements of $C(V)$, generated by the unit vectors in $V$. The morphism $\rho: \Pin(V) \to O(V)$ taking each unit vector to the corresponding reflection of $V$ is a double cover. Then $\Spin(V)=\Pin(V)^\circ$ is the inverse image of $\SO(V)$ under $\rho$.

 For $1 \leq j \leq n$, let $c_j(t)=te_j e_{-j}+t^{-1}e_{-j}e_j \in \Spin(V)$. This gives a morphism $c_j: \mb G_m \to \Spin(V)$. Define $c: \mb G_m^n \to \Spin(V)$ by 
  \beq
  c(t_1, \ldots, t_n)=c_1(t_1) \cdots c_n(t_n).
\eeq
The kernel of $c$ is 
\beq
\{(t_1, \ldots, t_n) \mid t_i= \pm 1, t_1 \cdots t_n=1\},
\eeq
and the image of $c$ is a maximal torus $\tilde T_V$ of $\Spin(V)$. 
The image of $\tilde T_V$ under $\rho$ is the subgroup of diagonal matrices in $\SO(V)$, relative to the basis of $V$ mentioned above.
More precisely, the restriction of $\rho$ to $\tilde T_V$ may be described by
\beq
\rho(c(t_1, \ldots, t_n))=\begin{cases} \diag(t_1^2, \ldots, t_n^2, t_n^{-2}, \ldots, t_1^{-2}) \\
 \diag(t_1^2, \ldots, t_n^2, 1, t_n^{-2}, \ldots, t_1^{-2}), \\
 \end{cases}
\eeq
depending on whether $\dim V=2n$ or $2n+1$.

The kernel of $\rho$ is generated by $z=c(-1,1,\ldots, 1)=-1 \in C(V)$. Pick $\sqrt{-1} \in F$, and put $c^+=c(\sqrt{-1},\sqrt{-1},\ldots, \sqrt{-1})$. Then $(c^+)^2=z^n$.

We now describe the center $Z$ of $\Spin(V)$.
\begin{enumerate}
\item When $\dim V=2n+1$, $Z$ is generated by $z$.  
\item When $\dim V=2n$, with $n$ odd, $Z$ is cyclic of order $4$, generated by $c^+$.
\item When $\dim V=2n$ with $n$ even, $Z$  is a Klein $4$-group generated by $z$ and $c^+$.
\end{enumerate}

 Define $\vt_i \in X^*(T_V)$  by
 \beq
\vt_i:\diag(t_1, \ldots, t_n, \ldots) \mapsto t_i.
\eeq
We identify $X^*(T_V)$  with $\Z^n$  through the bijection $\sum_i a_i \vt_i   \leftrightarrow (a_1, \ldots, a_n)$, and $X_*(T_V)$ with $\Z^n$ by $\nu \leftrightarrow (b_1, \ldots, b_n)$ when $\nu(t)=\diag(t^{b_1}, \ldots, t^{b_n}, \ldots)$.

Let $\Sigma$ be a set of weights formed by taking one representative from each  pair $\{ \vt_i,-\vt_i\}$.
 Then $\Sigma$ is a $\Z$-basis of $X^*(T_V)$. Of course, one choice is  $\Sigma_*=\{ \vt_1, \ldots, \vt_n \}$. Put $\omega_{\Sigma} = \sum_{\omega \in \Sigma} \omega$.

\begin{lemma} \label{blue.shirt} Let $d$ be a positive even integer, and $\zeta_d \in F^\times$ a primitive $d^{th}$ root of unity. Let $\nu \in X_*(T_V)$.
\begin{enumerate}
\item $\nu$ lifts to a cocharacter $\tilde \nu \in X_*(\tilde T_V) \Leftrightarrow \lip \omega_\Sigma,\nu \rip$ is even.
\item $\nu(\zeta_d)=1 \Leftrightarrow d \mid \lip \vt_i,\nu \rip$ for all $i$.
\item Assume  the conditions in (1) and (2) above. Then $\tilde \nu(\zeta_d)=1 \Leftrightarrow 2d \mid \lip \omega_\Sigma,\nu \rip$.
\end{enumerate}
\end{lemma}

\begin{proof} 

For the first statement, note that the image of  $X_*(\tilde T_V)$ in $X_*(T_V)$ is exactly $Q(T_V)$. One checks that $\lip \omega_\Sigma,\nu \rip$ is even iff $\nu \in Q(T_V)$.

For the second statement, just use that
\begin{equation} \label{nu.and.as}
\nu(t)=\diag(t^{b_1}, \ldots, t^{b_n}, \ldots),
\end{equation}
with $b_i=\lip \vt_i,\nu \rip$.

 Now consider the third statement for $\Sigma=\Sigma_*$. By hypothesis each $b_i$ in \eqref{nu.and.as} is even, and $\lip \omega_{\Sigma_*},\nu \rip=b_1+ \cdots + b_n$ is even.
Then
\beq
\tilde \nu(t)=c(t^{b_1/2}, \ldots, t^{b_n/2}), 
\eeq
so 
\beq
\tilde \nu(\zeta_d)=c \left(\zeta_d^{\frac{b_1}{2}}, \ldots, \zeta_d^{\frac{b_n}{2}} \right).
\eeq
Since each $b_i$ is even, each $\zeta_d^{\frac{b_i}{2}}=\pm 1$.
Therefore $\tilde \nu(\zeta_d)=1$, i.e., 
\beq
\zeta_d^{\frac{b_1+\cdots + b_n}{2}}=1,
\eeq
equivalently $2d$ divides $\lip \omega_{\Sigma_*},\nu \rip$, as claimed. Finally, by hypothesis $d$ divides each $\lip \vt_i,\nu \rip$, so that $\lip \omega_{\Sigma_*},\nu \rip \equiv \lip \omega_{\Sigma},\nu \rip \mod 2d$.
\end{proof}

 \section{Lifting Cocharacters} \label{s3}

We reformulate the lifting problem for an orthogonal representation in terms of its weights.  Throughout this section $G$ is a connected reductive group over an algebraically closed field $F$, and $T$ is a maximal torus of $G$.

Recall \cite{Spanier} that for nice topological spaces such as manifolds, if $\rho: \tilde Y \to Y$ is a covering map, then a continuous function  $\varphi: X \to Y$ lifts to $\hat \varphi: X \to \tilde Y$ iff $\varphi_*(\pi_1(X)) \leq \rho_*(\pi_1(\tilde Y))$ (with compatibly chosen basepoints on $X,Y,\tilde Y$). The purpose of the next proposition is to extend this to the setting of algebraic groups.

 \begin{lemma} \label{ref.asked} 
 Let $G,H$ be connected reductive groups, with maximal tori $T \leq G$ and $T_H \leq H$. Let $\varphi: G \to H$ be a morphism with $\varphi(T) \leq T_H$. 
 The induced map $\varphi_*: X_*(T) \to X_*(T_H)$ takes $Q(T)$ to $Q(T_H)$.
 \end{lemma}

\begin{proof}
Suppose first that $G,H$ are semisimple.
Write $\rho_G: G_{\simp} \to G$ for the universal cover, with maximal torus $T_{\simp}$ above $T$. Similarly we have $\rho_H: H_{\simp} \to H$, with maximal torus 
 $T_{H,\simp}$. 
Put $\Phi=\varphi \circ \rho_G: G \to H$.

Let  $\tilde G=(G_{\simp} \times_H H_{\simp})^\circ$, with projection maps $\pr_1: \tilde G \to G_{\simp}$ and $\pr_2:  \tilde G \to  H_{\simp}$.
It is easy to see that $\pr_1$ is a central isogeny; since $G_{\simp}$ is simply connected, it is an isomorphism by 2.15 of \cite{Springer.Corvallis}. 
If we put $\widetilde \varphi=\pr_2 \circ (\pr_1)^{-1}$, then the following diagram commutes:

\begin{equation} \label{lifting.diagram.1}
\xymatrix{
 G_{\simp} \ar[r]^{\widetilde \varphi} \ar[d]_{\rho_G} &H_{\simp} \ar[d]_{ \rho_H } \\
 G \ar[r]_{\varphi}    &H }
\end{equation}

Note that $\varphi$   restricts to a map from $T_{\simp}$ to $T_{H,\simp}$. Applying the covariant functor $T \mapsto X_*(T)$ gives the commutative square:

 \begin{equation*}
\xymatrix{
  X_*(T_{\simp}) \ar[r]^{\widetilde \varphi_*} \ar[d] &X_*(T_{H,\simp}) \ar[d] \\
 X_*(T) \ar[r]^{\varphi_*} &X_*(T_H) }
\end{equation*}

We have $X_*(T_{\simp})=Q(T)$ and similarly for $T_{H,\simp}$ by simple connectedness. The downward maps, being isogenies, take coroots to coroots and we deduce that 
$\varphi_*$ takes $Q(T)$ to $Q(T_H)$.     This conclusion also holds for  $G$ and $H$ connected reductive, for one applies the previous argument to the derived groups $G_{\der}$ and $H_{\der}$, recalling that the  coroots of $G$ lie in $X_*(T_{\der})$. \end{proof}

\begin{prop} \label{lift.proposition} Let $\rho: \tilde H \to H$ be a central isogeny of connected reductive  groups over $F$, and $\varphi: G \to H$ a morphism.  Pick a maximal torus $T_H \leq H$ containing $\varphi(T)$, and write $\varphi_*: X_*(T) \to X_*(T_H)$ for the induced map. 
Let $\tilde{ T}_H=\rho^{-1}(T_H) \leq \tilde H$, and write $\rho_*: X_*(\tilde T_H) \to X_*(T_H)$ for the induced map.  Then there exists a morphism $\hat \varphi: G \to \tilde H$ such that $\rho \circ \hat \varphi = \varphi$, iff $\im \varphi_* \subseteq \im \rho_*$. Moreover when this morphism exists, it is unique.
\end{prop}

\begin{proof}

Let   $\tilde G=(G \times_H \tilde H)^\circ$, with projection maps $\rho_G: \tilde G \to G$ and $\widetilde \varphi: \tilde G \to \tilde H$.  We have the   diagram:

\begin{equation} \label{lifting.diagram.1}
\xymatrix{
 \tilde G \ar[r]^{\widetilde \varphi} \ar[d]_{\rho_G} &\tilde H \ar[d]_{ \rho } \\
 G \ar[r]_{\varphi}  \ar@{-->}[ru]^{\hat \varphi} &H }
\end{equation}
 
 Put $\tilde T_H=\rho^{-1}(T_H)$  and $\tilde T=\rho_G^{-1}(T)$. 
Let us see the equivalence of the following statements:
 
 \begin{enumerate}
 
\item $\varphi$  lifts to $\hat \varphi: G \to \tilde H$.
 
\item $\widetilde \varphi$  factors through $\rho_G$.  
 
\item $\ker \rho_G \leq \ker \widetilde \varphi$. 

\item $\ker \rho_G \leq \ker \widetilde \varphi|_{\tilde T}$. 

\item $\tilde  \varphi|_{\tilde T}$ factors through $T$. 

\item $\varphi|_T$  lifts  to $\tilde T_H$. 

\item $\varphi_*$ lifts in the diagram:
 
 \begin{equation*}
\xymatrix{
 &X_*(\tilde T_H) \ar[d]^{ \rho_* } \\
 X_*(T) \ar[r]^{\varphi_*}  \ar@{-->}[ru]&X_*(T_H) }
\end{equation*}

\item $\im \varphi_* \leq \im \rho_*$.

\end{enumerate}
  
  For (1) $\Rightarrow$ (2), suppose $\varphi$ lifts to $\hat \varphi$. Then
\begin{align*} 
\rho (\widetilde \varphi (x)^{-1} \cdot  \hat \varphi( \rho_G (x))) =  1 \in H,
\end{align*} 
so the algebraic map $m: \tilde G \to \tilde H$ defined by 
\beq
m(x)= {\widetilde \varphi (x)}^{-1} \cdot \hat \varphi(\rho_G (x)),
\eeq
takes values in $\ker \rho$. Since $\ker \rho$ is discrete, and $\tilde G$ is connected, it must be that $m$ is constant. Thus $m(x)=1 \in \tilde H$   for all $x$, i.e., $\widetilde \varphi = \hat \varphi \circ \rho_G$.

For (2) $\Rightarrow$ (1), suppose $\widetilde \varphi=\hat \varphi \circ \rho_G$ for some morphism $\hat \varphi$.  From the identity $\varphi \circ \rho_G=\rho \circ \widetilde \varphi$ and the fact that $\rho_G$ is surjective we deduce that $\varphi=\rho \circ \hat \varphi$.

(2) $\Rightarrow$ (3) is immediate.

The implication (3) $\Rightarrow$ (2) follows from the universal property of $\tilde G/\ker \rho_G$. (See Section 5.5, page 92 of \cite{Springer}.)

Since $\ker \rho_G \leq Z(\tilde G) \leq \tilde T$, we have (3) $\Leftrightarrow$ (4).

The argument for (4) $\Leftrightarrow$ (5) is similar to the argument for (2) $\Leftrightarrow$ (3), and   (5) $\Leftrightarrow$ (6) is similar to   (1) $\Leftrightarrow$ (2).

To see (6) $\Leftrightarrow$ (7), note that the functors $T \mapsto X_*(T)$ and $L \mapsto L \otimes_{\Z} F^\times$ give an equivalence of categories between $F$-tori and free abelian groups of finite rank. 

  The equivalence (7) $\Leftrightarrow$ (8) is elementary.
  Thus (1)-(8) are equivalent.
  
  Finally, suppose that $\hat \varphi_1$ and $\hat \varphi_2$ are lifts of $\varphi$. Then $g \mapsto \hat \varphi_1(g) \hat \varphi_2(g)^{-1}$ is an algebraic map $G \to \ker \rho$ taking $1$ to $1$. Since $G$ is connected it must be that $\hat \varphi_1=\hat \varphi_2$.  
 \end{proof}

\begin{remark}
This proof did not use the property that $F$ has characteristic zero.
In the case of positive characteristic, it is sufficient for $F$ to be separably closed, and then $\hat \varphi$ is defined over $F$.  Suppose $F$ is an arbitrary field, and the maps $\rho$ and $\varphi$ are defined over $F$. By uniqueness,   $\hat \varphi$, when it exists, is fixed by the absolute Galois group of $F$ and hence defined over $F$.
\end{remark}

 By Lemma \ref{ref.asked},  $\varphi_*$ descends to  
\beq
\varphi_*: \pi_1(G) \to \pi_1(H).
\eeq
Again, since $\rho$ is an isogeny, we have $\rho_*(Q(\tilde T_H))=Q(T_H)$.  Therefore a lift $\hat \varphi$ in the diagram (\ref{lifting.diagram.1}) exists iff $\varphi_*$ lifts in the diagram:
 \begin{equation*}
\xymatrix{
 &\pi_1(\tilde H) \ar[d]^{ \rho_* } \\
 \pi_1(G) \ar[r]^{\varphi_*}  \ar@{-->}[ru]&\pi_1(H)}
\end{equation*}

Recall we have fixed a set $\ul \nu$ of cocharacters which generates $\pi_1(G)$.

 \begin{corollary}\label{lift.corollary}   A lift $\hat \varphi$ as in the above proposition exists iff $\varphi_*(\nu) \in \im \rho_*$ for each $\nu \in \ul \nu$.
  \end{corollary}

 \begin{defn} Let $(\varphi,V)$ be a representation of $G$.   For $\nu \in X_*(T)$, put
 \beq
 L_\varphi(\nu) =\sum_{\{\mu \in X^*(T) \mid \lip \mu,\nu \rip >0 \}} m_\varphi(\mu) \lip \mu,\nu \rip \in \Z.
 \eeq
 \end{defn}
 
 \begin{prop}  \label{L.phi.crit} Let $\varphi: G \to \SO(V)$ be an orthogonal representation.    
 For $\nu \in X_*(T)$, the cocharacter $\varphi_*(\nu) \in \im \rho_*$ iff $L_\varphi (\nu)$ is even.
Thus $\varphi$ is spinorial iff the integers $L_\varphi(\nu)$ are even for all $\nu \in \ul \nu$.
 \end{prop}
 (Compare Exercise 7 in Section 8, Chapter IX of \cite{Bou.Lie.7-9} and Lemma 3 in \cite{Prasad.Ramakrishnan}.)
 
 \begin{proof}
 We may assume that $\varphi(T) \leq T_V$.  
  By  Corollary \ref{lift.corollary}, $\varphi$ is spinorial iff $\varphi_*(\nu) \in \im \rho_*$ for all $\nu \in \ul \nu$.
   By Lemma \ref{blue.shirt}(1), we much check whether the integer $\lip \omega_\Sigma, \varphi_*(\nu) \rip$ is even for a suitable $\Sigma$.

 Write $P_V=\{ \pm \vt_i \mid 1 \leq i \leq n \}$, the weights of $V$ as a $T_V$-module. 
 Let
 \beq
 P_V^1=\{ \omega \in P_V \mid \lip \varphi^* \omega,\nu \rip \geq 0 \}.
 \eeq
 We may choose $\Sigma \subseteq P_V^1$ so that $\Sigma$ contains one representative from each pair $\{\vt_i,-\vt_i \}$ as in  Section \ref{tori.spin.section}.

 Then
 \beq
 \begin{split}
 \lip \omega_\Sigma, \varphi_*\nu \rip &=\sum_{\omega \in \Sigma} \lip \varphi^*\omega , \nu \rip \\
  &= \sum_{\{\mu \mid \lip \mu,\nu \rip >0 \}} \lip \mu,\nu \rip \cdot \dim V^{\mu} \\
  &= L_{\varphi}(\nu).\\
  \end{split}
 \eeq
   

 Thus $\varphi$ lifts iff  $L_{\varphi}(\nu)$ is even for all $\nu \in \ul \nu$. 
 \end{proof}

  Since $\varphi(Q(T)) \subseteq Q(T_V)$ we note:
  \begin{cor} \label{l.q} If $\nu \in Q(T)$, then $L_\varphi(\nu)$ is even.
  \end{cor}
 
 \bigskip
  
For two representations $\varphi_1,\varphi_2$, we have
\begin{equation} \label{add.in.phi}
L_{\varphi_1 \oplus \varphi_2}(\nu)=L_{\varphi_1}(\nu)+L_{\varphi_2}(\nu),
\end{equation}
since $m_{\varphi_1 \oplus \varphi_2}(\mu)=m_{\varphi_1}(\mu)+m_{\varphi_2}(\mu)$.

 \begin{cor}
 The adjoint representation of $G$ on $\mf g$ is spinorial iff $\delta \in X^*(T)$.
  \end{cor}
  
\begin{proof} If $\varphi$ is the adjoint representation, then
\beq
\begin{split}
L_\varphi (\nu) &= \sum_{\{\alpha \in R \mid \lip \alpha,\nu \rip >0\} }  \lip \alpha,\nu \rip \\
& \equiv \sum_{\alpha \in R^+} \lip \alpha,\nu \rip \mod 2 \\
&=2 \lip \delta,\nu \rip. 
\end{split}
\eeq
The corollary follows since the pairing $X^*(T) \times X_*(T) \to \Z$ is perfect.
\end{proof}

\begin{remark}
This is well-known; for $G$ a compact connected Lie group, see Section 5.56 of \cite{Adams}.  
\end{remark}

\begin{example} \label{pgl2.example}
 Let $G=\PGL_2$, with diagonal maximal torus $T$. Then $\pi_1(G)$ is generated by $\nu_0(t)=\diag(t,1)$ mod center. Let $\alpha$ be   the positive root  defined by 
$\alpha(\diag(a,b))=ab^{-1}$, and let $\varphi_j$ be the representation of $\PGL_2$ with highest weight $j \alpha$. Then
\beq
\begin{split}
L_{\varphi_j}(\nu_0) &= \lip \alpha, \nu_0 \rip+ \cdots + \lip j \alpha, \nu_0 \rip  \\
		&=1+ \cdots + j. \\
		\end{split}
		\eeq
		Therefore $\varphi_j$ is spinorial iff $j \equiv 0,3 \mod 4$.
 \end{example}

\section{Palindromy} \label{s4}

This section is the cornerstone of our paper.  The difficulty with determining the parity of $L_\varphi(\nu)$ is in somehow getting ahold of ``half" of the weights
of $V$, one for each positive/negative pair.  This amounts to knowledge of the polynomial part of a certain palindromic Laurent polynomial, and this we accomplish with a  derivative trick.

\begin{defn} For $(\varphi,V)$ a representation of $G$ and $\nu \in  X_*(T)$, consider the function $Q_{(\varphi,\nu)}: F^\times \to F$ defined by

\beq
\begin{split}
Q_{(\varphi,\nu)}(t) &=\Theta_\varphi(\nu(t)) \\
&=\tr(\varphi(\nu(t))). \\
\end{split}
\eeq
\end{defn}

If $\varphi$ is understood we may simply write `$Q_\nu(t)$'.  For $\gm \in T$, we have
\beq
\Theta_\varphi(\gm)=\sum_{\mu \in X^*} m_\varphi(\mu) \mu(\gm),
\eeq
so in particular
\begin{equation} \label{Qnu} 
Q_\nu(t)= \sum_{\mu \in X^*} m_\varphi(\mu) t^{\lip \mu,\nu \rip} \in \mb Z[t,t^{-1}].
\end{equation}

We note
\begin{itemize}
\item $Q_\nu(1)=\dim V$,
\item $Q_{\nu}'(1)= \sum_{\mu} m_\varphi(\mu) \lip \mu,\nu \rip$,
\item $Q''_{\nu}(1)= \sum_{\mu} \left(m_{\varphi}(\mu)\langle \mu,\nu \rangle^2 - m_{\varphi}(\mu)\langle \mu, \nu \rangle\right)$.
\end{itemize}

\begin{defn}  For $(\varphi,V)$ a representation of $G$ and $\nu \in  X_*(T)$, we set
\beq
q_{\varphi}(\nu)= \half Q''_{\nu}(1).
\eeq
\end{defn}

When $\varphi$ is self-dual, $m_\varphi(-\mu)=m_\varphi(\mu)$ for all $\mu \in X^*$, so in this case:

\begin{itemize}
\item $Q_\nu(t)=Q_\nu(t^{-1})$, i.e., $Q_\nu$ is ``palindromic",
\item $Q_\nu'(1)=0$,
\item  $Q_\nu''(1)=\sum_{\mu} m_\varphi(\mu) \lip \mu,\nu \rip^2 \in 2\Z$.
\end{itemize}

In particular, $q_\varphi(\nu)$ is an integer for all $\nu \in X_*$.  

\begin{lemma} \label{careful} For $\varphi$ self-dual and $\nu_1,\nu_2 \in X_*$, we have  
\beq
q_\varphi(\nu_1+\nu_2) \equiv q_\varphi(\nu_1)+q_\varphi(\nu_2) \mod 2.
\eeq
\end{lemma}

\begin{proof}
Breaking the sum over $\mu$ into a sum over nonzero pairs $\{ \mu,-\mu \}$ gives
\beq
\begin{split}
q_\varphi &=\half \sum_{\mu \in X^*} m_\varphi(\mu) \lip \mu,\nu \rip^2 \\
		&= \sum_{\{\mu,-\mu\}}m_\varphi(\mu) \lip \mu,\nu \rip^2 \in \Z. \\
\end{split}
\eeq
Therefore
\beq
\begin{split}
q_\varphi(\nu_1+\nu_2) &=\sum_{\{\mu,-\mu\}} m_\varphi(\mu) \left( \lip \mu,\nu_1 \rip^2+2 \lip \mu,\nu_1 \rip \lip \mu,\nu_2 \rip + \lip \mu,\nu_2 \rip^2 \right) \\
					& \equiv q_\varphi(\nu_1)+q_\varphi(\nu_2) \mod 2 \\
					\end{split}
					\eeq \end{proof}

Thus when $\varphi$ is self-dual, the function $q_\varphi: X_* \to \Z$ induces a group homomorphism $\ol q_\varphi: X_* \to \Z/2\Z$. Our goal in this section is to show that
\beq
\ol q_\varphi(\nu)=L_{\varphi}(\nu) \mod 2
\eeq
when $\varphi$ is orthogonal.

   Since $Q_\nu$ is palindromic, it may be expressed in the form 
\beq
Q_\nu(t)=H_\nu(t)+H_\nu(t^{-1})
\eeq
for a unique polynomial $H_\nu \in \Z[t]+\half \Z$.  Thus $H_\nu$ has integer coefficients, except its constant term may be half-integral.
More precisely,
\beq
H_\nu(t) = \sum_{\lip \mu,\nu \rip >0} m_\varphi(\mu) t^{\lip \mu,\nu \rip}+ \half \sum_{\lip \mu,\nu \rip =0} m_\varphi(\mu).
\eeq

What we want, at least mod $2$, is the integer
\beq
H'_\nu(1)=\sum_{\lip \mu,\nu \rip >0} m_\varphi(\mu) \lip \mu,\nu \rip=L_\varphi(\nu).
\eeq

By calculus we compute
\beq
Q_\nu''(1)=2(H_\nu'(1)+H_\nu''(1)).
\eeq
But $H''_\nu(1)$ is even!  This gives the crucial result:

\begin{prop} \label{jj.s2.not.good} If $\varphi$ is self-dual, then
\begin{equation}\label{LphiQnu}
L_\varphi(\nu) \equiv q_{\varphi}(\nu) \mod 2.
\end{equation} 
\end{prop}

\begin{corollary} \label{Qnugen.even}
Let $\varphi$ be an orthogonal representation of $G$.  Then $\varphi$ is spinorial iff $q_{\varphi}(\nu)$ is even for every $\nu \in \ul \nu$. 
\end{corollary}

\begin{proof}
This follows from Corollary \ref{lift.corollary}, Proposition \ref{L.phi.crit}, and the above equation.
\end{proof}

 \section{Irreducible Representations}  \label{s5}

In this section we compute $q_{\varphi}(\nu)$ when $\varphi$ is irreducible (not necessarily self-dual).  Our method follows the proof of Weyl's Character Formula in \cite{Goodman.Wallach}. For $\la \in X^*(T)^+$, write $(\varphi_\la,V_\la)$ for the irreducible representation of $G$ with highest weight $\la$.  For simplicity, we use the notation $q_\la$, $m_\la(\mu)$, etc. for
$q_{\varphi_\la}$, $m_{\varphi_\la}(\mu)$, etc. 

\subsection{Two Derivatives of Weyl's Character Formula}

For $\nu \in \mf t$, put
\beq
d_\nu= \prod_{\alpha \in R^+} \lip \alpha,\nu \rip,
\eeq
and for $\mu \in \mf t^*$, put
\beq
d_\mu=\prod_{\alpha \in R^+} \lip \mu, \alpha^\vee \rip.
\eeq

\begin{defn} Put
\beq
\mf t_{\reg}=\{\nu \in \mf t \mid d_\nu \neq 0\}.
\eeq
\end{defn}

Extend the function $q_\la: X_* \to \Z$ to the polynomial function \newline $q_\la: \mf t \to F$ defined by the formula
\beq
q_\la(\nu)=\half \sum_{\mu \in X^*} \lip \mu,\nu \rip^2 m_\la(\mu).
\eeq

We let $\Z[\mf t^*]$ denote the usual algebra of the monoid $\mf t^*$ with basis $e^{\mu}$ for $\mu \in \mf t^*$.  It contains the elements
\beq
J(e^\mu)=\sum_{w \in W} \sgn(w)e^{w \mu} \: \: \: \text{ and } \:  \: \: \ch(V_\la)=\sum_{\mu \in X^*} m_\la(\mu) e^{\mu}.
\eeq
  Recall the Weyl Character Formula (Prop. 5.10 in \cite{Jantzen}): 
  \beq
  \ch(V_\la) J(e^\delta)=J(e^{\la+ \delta}).
  \eeq

Write  $\eps: \Z[\mf t^*] \to \Z$ for the $\Z$-linear map so that $\eps(e^\mu)=1$ for all $\mu \in \mf t^*$ (i.e., the augmentation); it is a ring homomorphism.  Given $\nu \in \mf t$, write $\frac{\partial}{\partial \nu}: \Z[\mf t^*] \to \Z[\mf t^*]$ for the $\Z$-linear map so that $\dnu(e^\mu)=\lip \mu,\nu \rip e^\mu$; it is a $\Z$-derivation.  
  Note that $\eps \left(\ch(V_\la)\right)=\dim V_\la$, and
\begin{equation} \label{2nd.derivative}
\left(\eps \circ \frac{\partial^2}{\partial \nu^2}\right) \ch(V_\la)= Q''_\nu(1).
 \end{equation}

\begin{prop} \label{bugs}  For $\nu \in \mf t_{\reg}$, we have
\beq
q_\lambda(\nu)=\frac{\sum_{w \in W} \sgn(w) \lip w(\la+\delta),\nu \rip^{N+2}}{(N+2)! d_\nu} - \frac{1}{48} \dim V_\la |\nu|^2,
\eeq
where $N=|R^+|$.
 \end{prop}
 
 \begin{proof}
We apply $\eps \circ \Dnu{N+2}$ to both sides of $J(e^{\la+ \delta})=\ch(V_\la) J(e^\delta)$.  On the left we have
 \begin{equation} \label{left1}
 \left(\eps \circ \Dnu{N+2} \right) J(e^{\la+ \delta})=\sum_{w \in W} \sgn(w) \lip w(\la+\delta),\nu \rip^{N+2}.
 \end{equation}
 The right hand side requires more preparation.
For $\alpha \in R^+$, let $r_\alpha=e^{\alpha/2}-e^{-\alpha/2}$.  Then
\begin{itemize}
\item $\eps(r_\alpha)=0$,
\item $\eps \circ \dnu (r_\alpha)=\lip \alpha,\nu \rip$,
\item $\Dnu{2} r_\alpha=\frac{1}{4} \lip \alpha,\nu \rip^2 r_\alpha$,
\item $J(e^\delta)=\prod_{\alpha \in R^+} r_\alpha$.
\end{itemize}
 
 The last equality is a familiar identity from \cite{Bou.Lie.4-6}.  We may now apply the following lemma:

 \begin{lemma} Let $R$ be a commutative ring, $D: R \to R$ a derivation, and $\eps: R \to R'$ a ring homomorphism.  Suppose that $r_1, \ldots, r_N \in \ker \eps$.  Then
 \begin{enumerate}
 \item $\eps(D^n(r_1 \cdots r_N))=0$ for $0 \leq n< N$.
 \item $\eps(D^N(r_1 \cdots r_N))=N! \prod_{i=1}^N \eps(D(r_i))$.
 \item If also $D^2(r_i) \in \ker \eps$ for all $i$ then 
  $\eps(D^{N+1}(r_1 \cdots r_N)) = 0$. 
 \item Suppose further that there are $c_i \in R$ so that $D^2(r_i)=c_i r_i$.  Then
 \beq
 \eps(D^{N+2}(r_1 \cdots r_N))=\frac{(N+2)!}{6} \left( \prod_i \eps(D(r_i)) \right) \left( \sum_i c_i \right).
 \eeq
 \end{enumerate}
 \end{lemma}
 
 \begin{proof}
 This follows from the Leibniz rule for derivations:
 \begin{equation*} 
 D^n(r_1 \cdots r_k)=\sum_{i_1+\cdots+i_k=n} \binom{n}{i_1, \ldots, i_k} D^{i_1}(r_1) \cdots D^{i_k}(r_k).
 \end{equation*}
\end{proof}

Thus in our case,
 \begin{enumerate}
\item $ (\eps \circ \Dnu{n})J(e^\delta) =0$ for $0 \leq n< N$,
\item $ (\eps \circ \Dnu{N})J(e^\delta) =N! d_{\nu}$,
\item $ (\eps \circ \Dnu{N+1})J(e^\delta)= 0$,
\item  $ (\eps \circ \Dnu{N+2})J(e^\delta)=\frac{(N+2)!}{24} d_{\nu}  \sum_{\alpha>0} \lip \alpha,\nu \rip^2$.
 \end{enumerate}
 
Now we are ready to consider
\beq
  \left(\eps \circ \Dnu {N+2} \right)  (\ch(V_\la) J(e^\delta)).
\eeq
Applying the Leibniz rule to the above gives
\beq
\binom{N+2}{2}Q''_\nu(1) N! d_{\nu} + \dim V_\la \frac{(N+2)!}{24} d_{\nu}  \sum_{\alpha>0} \lip \alpha,\nu \rip^2.
\eeq
Equating this with (\ref{left1}) yields the identity
\begin{equation} \label{whence}
\sum_{w \in W} \sgn(w) \lip w(\la+\delta),\nu \rip^{N+2}=(N+2)! d_\nu \left(q_\la(\nu)+ \frac{\dim V_\la}{24} \sum_{\alpha>0} \lip \alpha,\nu \rip^2 \right),
\end{equation}
whence the proposition.
  \end{proof}

\subsection{Anti-$W$-invariant Polynomials} $\text{                        }$

The expression  ``$\sum_{w \in W} \sgn(w) \lip w(\la+\delta),\nu \rip^{N+2}$" in our formula demands simplification.  This can be done by applying the theory of anti-$W$-invariant polynomials.

Let $f: \mf t \to F$ be a polynomial function.  We say that $f$ is \emph{anti-$W$-invariant}, provided for all $w \in W$ and $\nu \in \mf t$ we have
\beq
f(w(\nu))=\sgn(w) f(\nu).
\eeq
The polynomial $\nu \mapsto d_\nu$ is a homogeneous anti-$W$-invariant polynomial of degree $N$.
According to  \cite{Bou.Lie.4-6}, page 118, if $f$ is a homogeneous anti-$W$-invariant polynomial of degree $d$, then  there exists a homogeneous $W$-invariant polynomial $p: \mf t \to F$ so that $f(\nu)=p(\nu) d_\nu$.  Necessarily $d=\deg f \geq N$ and $p$ has degree $d-N$.  Similarly, if $g: \mf t^* \to F$ is a homogeneous anti-$W$-invariant polynomial, then $g(\mu)=p(\mu) d_\mu$ for a $W$-invariant polynomial $p$ on $\mf t^*$.

In this section we will make use of the famous Weyl dimension formula, which we recall is $\dim V_\la=d_{\la+\delta}/d_\delta$.

\begin{defn} Let  $k$ be  a nonnegative integer.  Put
\beq
F_k(\mu,\nu)=\sum_{w \in W} \sgn(w) \lip w(\mu),\nu \rip^k,
\eeq
for $\mu \in \mf t^*$ and $\nu \in \mf t$.
\end{defn}

\begin{prop} \label{F_k} Let  $\mf g$ be simple.  
Then 
\beq
F_k(\mu,\nu)= \begin{cases} 
0 & \text{ if } 0 \leq k < N \text{ or } k=N+1,\\
N!\cdot \dfrac{d_\mu d_\nu}{d_\delta}& \text{ if } k=N, \\ 
 \dfrac{(N+2)!}{48   |\delta|^2} \cdot \dfrac{d_\mu d_\nu}{d_\delta} |\mu|^2|\nu|^2  & \text{ if } k=N+2. \\
\end{cases}
\eeq
 \end{prop}
 
 \begin{proof}  Each $F_k$ may be viewed as a polynomial in two ways: as a function of $\mu$ and as a function of $\nu$. 
It is either identically $0$, or homogeneous of degree $k$.  Both the functions $\mu \mapsto F_k(\mu,\nu)$ and $\nu \mapsto F_k(\mu,\nu)$ are anti-$W$-invariant.  Therefore $F_k(\mu,\nu)$ either vanishes, or is the product of $d_\mu d_\nu$ and a homogeneous $W$-invariant polynomial of degree $k-N$ in both  $\nu$ and $\mu$.  By degree considerations, $F_k$ must vanish for $0 \leq k < N$. 
  
  \bigskip
  
 {\bf Case $k=N$:}  Here $F_N(\mu,\nu)=c d_\mu d_\nu$ for some constant $c \in F$, independent of $\mu$ and $\nu$.  To determine $c$,
 we apply $\eps \circ \Dnu{N}$ to both sides of $J(e^{\la+ \delta})=\ch(V_\la) J(e^\delta)$.  On the left we have
 \begin{equation} \label{left}
\left( \eps \circ \Dnu{N} \right) J(e^{\la+ \delta})=F_N(\la + \delta,\nu).
 \end{equation}
 On the right we proceed as in the proof of Proposition \ref{bugs} to obtain $N! \cdot \dim V_\la \cdot d_\nu$.  Therefore
 \beq
 c \cdot d_{\la+\delta} d_\nu = N! \cdot \dim V_\la \cdot d_\nu,
 \eeq
so that $c=\dfrac{N!}{d_\delta}$.
  
 \bigskip 

{\bf Case $k=N+1$:}
 Since $\mf g$ is simple, both $\mf t$ and ${\mf t}^*$ are irreducible representations of $W$.  
  If $\dim \mf t>1$, there is no $1$-dimensional invariant subspace.  When $\dim \mf t=1$, $W$ acts by a nontrivial reflection.  Therefore there is no   $W$-invariant vector, i.e., no $W$-invariant polynomial of degree $1$.  Thus in all cases $F_{N+1}$ vanishes.

 {\bf Case $k=N+2$:} 
Let us write $F_{N+2}(\mu,\nu)=\mc Q_\mu(\nu) d_\nu$ with $\mc Q_\mu$ a $W$-invariant quadratic form on $\mf t$.    The corresponding bilinear form on $\mf t$ is $W$-invariant; as $\mf t$ is an irreducible $W$-representation, this bilinear form must be a scalar multiple of the Killing form.  Thus we may write 
 \begin{equation} \label{cR.def}
 F_{N+2}(\mu,\nu)=c_{R} d_\mu d_\nu |\mu|^2|\nu|^2;
 \end{equation}
  it remains to determine $c_{R}$.
 
Let $\sigma$ be as in Section \ref{pairings.section}. Employing \cite{Bou.Lie.7-9}, Ch. VIII, Section 9, Exercise 7, we obtain the value at $\nu=\sigma(\delta) \in \mf t$:
\beq
\begin{split}
Q_{\sigma(\delta)}''(1) &=\sum_{\mu} \lip \mu,\sigma (\delta) \rip^2 m_{\la}(\mu) \\
&=\frac{\dim V_\la}{24} \cdot (\la,\la+2\delta).
\end{split}
\eeq
 
Substituting this into (\ref{whence}) gives
\beq
\begin{split}
F_{N+2}(\la+\delta,\sigma(\delta)) &=  \half d_{\sigma(\delta)} (N+2)! \left(Q_{\sigma(\delta)}''(1)+ \frac{\dim V_\la}{24} |\delta|^2 \right) \\
&=  d_{\sigma(\delta)} (N+2)! \frac{\dim V_\la}{48} |\la+\delta|^2.\\
\end{split}
\eeq
On the other hand, from (\ref{cR.def}) we have
\beq
\begin{split}
F_{N+2}(\la+\delta,\sigma(\delta)) &=c_R d_{\la+\delta} d_{\sigma(\delta)} |\la+\delta|^2 |\delta|^2 \\
				&=c_R \dim V_\la d_\delta  d_{\sigma(\delta)} |\la+\delta|^2 |\delta|^2.\\
\end{split}
\eeq
We deduce that
\beq
c_R=\frac{(N+2)!}{48 d_\delta |\delta|^2}.
\eeq

The proposition follows from this. \end{proof}
 
For the general case,  say $\mf g=\mf g^1 \oplus \cdots \oplus \mf g^\ell \oplus \mf z$ with each $\mf g^i$ simple, and $\mf z$ abelian.  A Cartan subalgebra $\mf t \subset \mf g$ is the direct sum of the center $\mf z$ and Cartan subalgebras $\mf t^i \subset \mf g^i$, and  the Weyl group $W=W(\mf g,\mf t)$ is the direct product of the Weyl groups $W^i=W^i(\mf g^i,\mf t^i)$. Any $\mu \in \mf t^*$ is equal to $\mu^z+ \sum_i \mu^i$ with $\mu^i \in (\mf t^i)^*$ and $\mu^z \in \mf z^*$; similarly for $\nu \in \mf t$.
  Let $N_i$ (resp. $N$) be the number of positive roots in $\mf g^i$ (resp. $\mf g$).

 \begin{prop} Let $\mu \in \mf t^*$ and $\nu \in \mf t$, with notation as above.  Then

 \beq
 F_k(\mu,\nu)= \begin{cases} 
0 & \text{ if } 0 \leq k < N,\\
N!\cdot \dfrac{d_\mu d_\nu}{d_\delta}& \text{ if } k=N, \\ 
(N+1)! \cdot  \dfrac{d_\mu d_\nu}{d_\delta}  \lip \mu^z,\nu^z \rip & \text{ if } k=N+1,\\
 \dfrac{(N+2)!}{48} \cdot \dfrac{d_\mu d_\nu}{d_\delta} \sum_i \dfrac{|\mu^i|^2 |\nu^i|^2}{|\delta^i|^2}+ 
   \dfrac{(N+2)!}{2} \cdot \dfrac{d_\mu d_\nu}{d_\delta} \lip \mu^z,\nu^z \rip^2 & \text{ if } k=N+2. \\
\end{cases}
\eeq
 \end{prop}
  
\begin{proof}

   If $\mf z=0$, we have
\beq
\begin{split}
 F_{k}(\mu,\nu) &= \sum_{w \in W} \sgn(w) \lip w(\mu^1+ \cdots+ \mu^\ell) , \nu^1 + \cdots + \nu^\ell \rip^{k} \\
 		&= \sum_{w=(w_1, \ldots, w_\ell) \in W} \sgn(w) \left( \sum_{i=1}^\ell \lip w_i(\mu^i),\nu^i \rip \right)^{k} \\
		&= \sum_w \sgn(w) \sum_{k_1+ \cdots + k_\ell=k} \binom{k}{k_1, \ldots, k_\ell} \prod_i \lip w_i(\mu^i),\nu^i \rip^{k_i} \\
		&= \sum_{k_1+ \cdots + k_\ell=k} \binom{k}{k_1, \ldots, k_\ell} \prod_i \sum_{w_i \in W^i} \sgn(w_i) \lip w_i(\mu^i),\nu^i \rip^{k_i} \\
		&= \sum_{k_1+ \cdots + k_\ell=k} \binom{k}{k_1, \ldots, k_\ell}  \prod_i F_{k_i}(\mu^i,\nu^i). \\
 \end{split}
 \eeq
The product $ \prod_i F_{k_i}(\mu^i,\nu^i)$ vanishes unless $k_i \geq N_i$ for all $i$.  So $F_k(\mu,\nu)$ vanishes for $k<N$.

Now put $k=N+2$.  Since $k_1+ \cdots + k_\ell=N+2$, we see by Proposition \ref{F_k} that this product is only nonzero when some $k_i=N_i+2$ and the other $k_i$ equal $N_i$.
Therefore
 \beq
 \begin{split}
 F_{N+2}(\mu,\nu) &=\sum_{i=1}^\ell \binom{N+2}{N_1, \ldots, N_i+2, \ldots, N_\ell} F_{N_1}(\mu^1, \nu^1) \cdots F_{N_i+2}(\mu^i,\nu^i) \cdots F_{N_\ell}(\mu^\ell,\nu^\ell) \\
 			&=   \frac{(N+2)!}{48} \cdot \frac{d_\mu d_\nu}{d_\delta} \sum_{i=1}^\ell \frac{|\mu^i|^2 |\nu^i|^2}{|\delta_i|^2}.\\
\end{split}
 \eeq  
 If $\mf z \neq 0$, there is an extra term $\dfrac{(N+2)!}{2} \cdot \dfrac{d_\mu d_\nu}{d_\delta} \lip \mu^z,\nu^z \rip^2$.   
 The other cases are similar.
 \end{proof}
 
 \subsection{Main Theorem for $\varphi$ Irreducible} \label{MT.irred}

  \begin{prop} \label{this.cor} Let $\mf g$ be simple and $\varphi=\varphi_\la$ irreducible.  Then for all $\nu \in \mf t$, we have
 \beq
q_{\lambda}(\nu)=\frac{ \dim V_\la \cdot \chi_\la(C)}{\dim \mf g} \cdot \frac{ |\nu|^2}{2}.
 \eeq
 \end{prop}

 \begin{proof}
Let  $\nu \in \mf t_{\reg}$.  By Proposition \ref{bugs},
 \beq
 \begin{split}
   q_\la(\nu) &= \frac{F_{N+2}(\la+\delta,\nu)}{(N+2)! d_\nu}-\frac{1}{48} \dim V_\la |\nu|^2 \\
  			&=  \frac{1}{48 |\delta|^2} \cdot \frac{d_{\la+\delta}}{d_\delta} |\la+\delta |^2|\nu|^2-\frac{1}{48} \dim V_\la |\nu|^2\\
			&= \frac{1}{48 |\delta|^2} \dim V_\la |\nu|^2 \left(|\la+\delta|^2-|\delta|^2 \right). \\	
 \end{split}
 \eeq
 Recall that $\chi_\la(C)=|\la+\delta|^2-|\delta|^2$.  Moreover, by \cite{Bou.Lie.7-9}, Exercise 7, page 256, we have $|\delta|^2=\dim \mf g/24$. These substitutions give the proposition for the case $\nu \in \mf t_{\reg}$; by continuity it holds for $\nu \in \mf t$.
 \end{proof}
 

\bigskip

\begin{example}
Revisiting $\PGL_2$ from Example \ref{pgl2.example}, one computes $|\nu_0|^2=2$,  $\dim V_{j \alpha}=2j+1$, and $\chi_{j \alpha}=\half (j^2+j)$, so 
\beq
q_{j \alpha}(\nu_0)=\frac{j(j+1)(2j+1)}{6}.
\eeq
So as before $\varphi_{j \alpha}$ is spinorial iff $j \equiv 0,3 \mod 4$.

\end{example}

The   case of $G$ reductive is similar:

\begin{prop} \label{ss.Qnu}   With notation as before, and $\varphi_\la$ irreducible, we have
 \beq
q_{\lambda}(\nu)= \half \dim V_\la \cdot \sum_i \frac{|\nu^i|^2 \chi_{\la^i}(C^i)}{ \dim \mf g^i}.
 \eeq
 \end{prop}

\begin{proof} For $\mf z=0$, we have
\beq
 \begin{split}
  q_\la(\nu) &= \frac{F_{N+2}(\la+\delta,\nu)}{(N+2)! d_\nu}-\frac{1}{48} \dim V_\la |\nu|^2 \\
			&=\frac{1}{48} \cdot \dim V_\la \sum_{i=1}^\ell \frac{|\lambda^i+\delta^i|^2 |\nu^i|^2}{|\delta^i|^2} -\frac{1}{48} \dim V_\la \sum_i |\nu_i|^2 \\
			&= \frac{1}{48} \dim V_{\lambda}\sum_{i=1}^l |\nu_i|^2 \left(\frac{|\lambda^i + \delta^i|^2 - |\delta^i|^2 }{|\delta^i|^2} \right).
			\end{split}
			\eeq
			
			The substitution $|\delta^i|^2 = \dim \mf g^i /24$, gives the proposition in the semisimple case.
If $\mf z \neq 0$, one must add $\half \lip \la,\nu^z \rip^2 \cdot  \dim V_\la$.  However for $\varphi_\la$ irreducible orthogonal, necessarily $\la$ annihilates the center.
			\end{proof}

 \begin{cor} \label{irred.not.simple} An irreducible orthogonal representation $\varphi_\la$ of $G$ is spinorial iff  
  \beq
 \half \dim V_\la \sum_i \frac{|\nu^i|^2 \chi_{\la^i}(C^i)}{ \dim \mf g^i}
 \eeq
 is even for all cocharacters $\nu \in \ul \nu$.
 \end{cor}

\begin{proof} This follows from Proposition \ref{ss.Qnu}  and Corollary \ref{Qnugen.even}.
\end{proof}

\begin{example} \label{so.4.example}
 For $G=\SO_4$, the Lie algebra $\mf g$ is not simple.  Here, $X^*(T)=X_{\sd}=X_{\orth}$, where $T$ is the diagonal torus of $G$.

We may identify  $\Spin_4 \to \SO_4$ with the cover $\SL_2 \times \SL_2 \to \SO_4$ as in Exercise 7.16 of \cite{Fulton.Harris}.
In particular, we may identify $\mf g$ with the Lie algebra of ${\mf s \mf l}_2 \times \mf s \mf l_2$, and $\mf t$ with pairs of diagonal matrices in  ${\mf s \mf l}_2 \times \mf s \mf l_2$.
The irreducible representations  of $\SL_2 \times \SL_2$ are the external tensor products
$V_{a,b}=\Sym^a V_0 \boxtimes \Sym^b V_0$, where $V_0$ is the standard $2$-dimensional representation of $\SL_2$.   Here $a,b$ are nonnegative integers; the representation $V_{a,b}$ descends to a representation $\varphi_{a,b}$ of $G$ when $a \equiv b \mod 2$.

Let $\nu_s=\diag(s,-s)$; then $\nu_{s,t}=(\nu_s,\nu_t) \in \mf t$ corresponds to a cocharacter of $T$ iff either $s,t \in \Z$, or $2s$ and $2t$ are both odd integers.   Proposition \ref{ss.Qnu} gives
\beq
\begin{split}
q_{\varphi_{a,b}}(\nu_{s,t}) &=\half (a+1)(b+1) \left( \frac{4s^2 \cdot \frac{1}{4}a(a+2)}{3}+\frac{4t^2 \cdot \frac{1}{4} b(b+2)}{3} \right) \\
&= s^2 (b+1) \binom{a+2}{3}+t^2(a+1) \binom{b+2}{3}. \\
\end{split}
\eeq

Since $\pi_1(G)$ is generated by $\nu_{\half, \half}$, we deduce that $V_{a,b}$ is spinorial iff
 \beq
  (b+1) \binom{a+2}{3}+(a+1) \binom{b+2}{3}, \\
  \eeq
  which is always a multiple of $4$, is divisible by $8$.
\end{example}

By the following, spinoriality for irreducible orthogonal representations of connected reductive groups reduces to the semisimple case:
 
\begin{prop} \label{mod.z.spin} Let $G$ be a connected  reductive group  and $\varphi: G \to \SO(V)$ an irreducible orthogonal representation.  Then $\varphi$ factors through the quotient $p: G \to G/Z(G)^\circ$, so that $\varphi= \varphi' \circ p$ with
$\varphi': G/Z(G)^\circ \to \SO(V)$.  Moreover $\varphi$ is spinorial iff $\varphi'$ is spinorial. \end{prop}
 
\begin{proof}
By Schur's Lemma,  $\varphi(Z(G))$ is a subgroup of the scalars in $\SO(V)$, namely $\{ \pm \id_V\}$.  Therefore $\varphi(Z(G)^\circ)$ is trivial.  This gives the first part, and the second part is similar.
\end{proof}

\subsection{Existence of Spinorial Representations}
 
We continue with $G$ connected reductive. Let $a$ be a positive multiple of $4$ and $\la=a \delta$. Consider the irreducible representation $(\varphi_\la,V_\la)$. 
 It is easy to see that $V_\la$ is orthogonal, and  $\dim V_\la=(a+1)^N$.
 Therefore for $\nu \in X_*(T)$, we have by Proposition \ref{irred.not.simple}:
 \beq
 \begin{split}
 q_\la(\nu) &=\half (a+1)^N \sum_{i} \frac{|\nu^i|^2 ( \la_i,\la_i+2 \delta_i )}{\dim \mf g^i} \\ 
 &=\frac{1}{24}(a+1)^N a(a+2) \frac{|\nu|^2}{2}, \\
 \end{split}
 \eeq
 since $(\delta_i,\delta_i)=\dfrac{ \dim \mf g_i}{24}$ for each $i$.
 
 Therefore  $\varphi_\la$ is spinorial iff the quantity
  
  \beq
  \frac{p(\ul \nu)}{24} (a+1)^N \cdot a(a+2)
 \eeq
 is even. From this we deduce:
 \begin{enumerate}
 \item The representation $V_{\la}$ is spinorial when $a \equiv 0 \mod 8$.
 \item If $p(\ul \nu)$ is even, then $V_{\la}$ is spinorial.
\end{enumerate}

In particular:

\begin{cor} A nonabelian connected reductive group has a nontrivial irreducible  spinorial representation.
\end{cor}
\begin{proof} By the above, one may   take $\la=8 \delta$. \end{proof}

  \section{Reducible Representations} \label{reducible.rep.section}
  
  In this section we treat the case of $\varphi$ orthogonal, but not necessarily irreducible.

\subsection{Spinoriality of $\varphi \oplus \varphi^\vee$} \label{phi.plus.dual}
   For a representation $(\varphi,V)$ of a connected reductive group $G$, consider the orthogonal representation $(S(\varphi), V \oplus V^\vee)$ defined as follows.  We give $V \oplus V^\vee$ the quadratic form  
\beq
\mc Q((v,v^*))=\lip v^*,v \rip,
\eeq
and write $S(\varphi)$ for the representation of $G$ on $V \oplus V^\vee$ given by
\beq
{}^g(v,v^*)=(\varphi(g)v,\varphi^\vee(g)v^*).
\eeq

For $\nu \in X_*(T)$, $\mu \in X^*(T)$, and $t \in F$, $\nu(t)$ acts on $V^\mu$ by the scalar $t^{\lip \mu,\nu\rip}$.
Therefore we have
\begin{equation} \label{det.s}
\det \varphi(\nu(t))=t^{s_\varphi(\nu)}, 
\end{equation}
where
\beq
s_\varphi(\nu)=\sum_{\mu \in X^*(T)} m_\varphi(\mu) \lip \mu,\nu \rip.
\eeq

\begin{prop} \label{l.vs.s(nu)} $L_{S(\varphi)}(\nu) \equiv s_\varphi(\nu) \mod 2$.  Therefore $S(\varphi)$ is spinorial iff $s_\varphi(\nu)$ is even for all $\nu \in \ul \nu$. If $G$ is semisimple then $S(\varphi)$ is spinorial.
\end{prop}

\begin{proof}
Since $m_{\varphi^\vee}(\mu)=m_{\varphi}(-\mu)$, we have
\beq
\begin{split}
L_{S(\varphi)}(\nu) &= \sum_{\{\mu \mid \lip \mu,\nu \rip >0\}} (m_{\varphi}(\mu) +m_{\varphi}(-\mu) )\lip \mu,\nu \rip \\	
			&\equiv  \sum_{\{\mu \mid \lip \mu,\nu \rip >0\}} (m_{\varphi}(\mu) - m_{\varphi}(-\mu)) \lip \mu,\nu \rip \mod 2 \\
			&=s_\varphi(\nu). \\
			\end{split}
			\eeq
			
			 When $G$ is semisimple, the image of $\varphi$ lies in $\SL(V)$, and so
$s_\varphi(\nu)=0$.  Therefore $L_{S(\varphi)}(\nu)$ is even and so $S(\varphi)$ is spinorial in this case.
			
			 \end{proof}

  Now, assume $\varphi=\varphi_\la$ is irreducible.  
   Let $\nu = \nu^{z} + \nu'$ correspond to the decomposition $\mf{t} = \mf{z} \oplus \mf{t'}$.  
\begin{thm}  \label{Sum.spin.thm} $S(\varphi_\la)$ is spinorial iff the integers
\beq
\lip \la, \nu^z \rip \cdot \dim V_{\la}
\eeq
are even for all $\nu \in \ul \nu$.
\end{thm}
  
\begin{proof}  

Differentiating both sides of (\ref{det.s})   at $t=1$ gives
\beq
s_\varphi(\nu)=\tr d\varphi (d\nu(1)),
\eeq
where $\tr: \mf t_V \to F$ is the  trace.

  Write $\mf z_V$ for the center of the Lie algebra of $\GL(V)$, and  $\mf t'_V$ for the Lie algebra of the maximal torus in $\SL(V)$.  We have a direct sum decomposition $\mf t_V=\mf t_V' \oplus \mf z_V$, and similarly for $\mf t$.   Let $\pr_V: \mf t_V \to \mf z_V$  and $\pr: \mf t \to \mf z$ be the projections.   

Note that the diagram
 \beq
 \xymatrix{
F \ar[r]^{d\nu} & \mf t \ar[r]^{d\varphi} \ar[d]_{\pr} & \mf t_V \ar[r]^{\tr} \ar[d]_{\pr_{V}} & F \\
                      &     \mf z \ar[r]^{d\varphi} & \mf z_V \ar[ru]^{\tr}    &                  }
\eeq
 is commutative.  Moreover  $\tr( d\varphi(z))=d \la(z) \cdot \dim V_\la$ for $z \in \mf z$, by Schur's Lemma.
   It follows that
 \beq
 s_\varphi(\nu)= d\la(\nu^z) \cdot \dim V_\la,
 \eeq
 so the theorem follows from the previous proposition.
\end{proof}
  
  \begin{example}
Let $G=\GL_2$.  We may parametrize $X^*(T)^+$ with integers $(m,n)$ with $0 \leq n \leq m$ via:
 \beq
 \lambda_{m,n} \mat {t_1}{}{}{t_2}=t_1^m t_2^n.
 \eeq
Let $\nu_0(t)=\mat t001$, so that $(\nu_0)^z=\half(1,1)$.
  Then $\dim V_{\lambda_{m,n}}=m-n+1$ and  $\lip \la,\nu_0^z \rip=\half (m+n)$, so $s_{\lambda_{m,n}}(\nu_0) =\half (m+n)(m-n+1)$.  From Theorem \ref{Sum.spin.thm}, we deduce that the representation $S(\varphi_{\lambda_{m,n}})$ of $\GL_2$ is spinorial iff the integer $\half (m+n)(m-n+1)$ is even.
\end{example}
  
 \subsection{General Lifting Condition}  
 
 We begin this section by gathering our results to give a general lifting condition for reducible orthogonal representations.

Recall we have   $\mf g=\mf g^1 \oplus \cdots \oplus \mf g^\ell \oplus \mf z$ with each $\mf g^i$ simple, and $\mf z$ abelian.   Thus our $\nu \in \mf t$ decomposes into $\nu^z+ \sum_i \nu^i$ with $\nu^i \in \mf t^i$ and $\nu^z \in \mf z$.
  
  \begin{prop} \label{Samelson.book} If $\varphi$ is an orthogonal representation of $G$, then $\varphi$ is a direct sum of representations of the following type:
 \begin{itemize}
 \item Irreducible orthogonal representations.
 \item The representations $S(\sigma)$, with $\sigma$ irreducible.
 \end{itemize}
 \end{prop}
 
\begin{proof} This follows from Lemma C in Section 3.11 of \cite{Samelson}.
\end{proof}

\begin{thm} \label{general.lifting}
Let $\varphi=S(\sigma) \oplus \bigoplus_j \varphi_j$, with each $\varphi_j$  irreducible orthogonal with highest weight $\la_j$, and $\sigma=\bigoplus_k  \sigma_k$, with each $\sigma_k$ irreducible with highest weight $\gm_k$.  Then $\varphi$ is spinorial iff for all $\nu \in \ul \nu$, the  integer
\beq
q_\varphi(\nu)=\sum_k \lip \gm_k,\nu^z \rip \cdot \dim V_{\gm_k} + \sum_i \frac{|\nu^i|^2}{2} \sum_j \frac{\dim V_{\la_j} \cdot \chi_{\la_j^i}(C^i)}{\dim \mf g^i}
\eeq
is even.
\end{thm}

 \begin{proof}
 We have
 \beq
 \begin{split}
 L_\varphi(\nu) &= \sum_k L_{S(\sigma_k)}(\nu)+ \sum_j L_{\varphi_j}(\nu) \\
 			& \equiv \sum_k s_{\sigma_k}(\nu) + \sum_j q_{\varphi_j}(\nu) \mod 2. \\
		\end{split}
		\eeq
 The first equality is by (\ref{add.in.phi}), and the congruence is by Proposition \ref{l.vs.s(nu)} and Proposition \ref{jj.s2.not.good}.
 The conclusion then follows from Theorem \ref{Sum.spin.thm} and Corollary \ref{irred.not.simple}.
 \end{proof}

 Note that when $G$ is semisimple, the sum over $k$ vanishes.

 \subsection{Case of $\mf g$ Simple} The situation is much nicer when $\mf g$ is simple; let us deduce Theorem \ref{mthm} as a Corollary of Theorem  \ref{general.lifting}:

 \begin{proof} (of Theorem \ref{mthm})
 By Theorem \ref{general.lifting}, we have
 \beq
  \begin{split}
 q_\varphi(\nu) &=  \frac{|\nu|^2}{2} \sum_j \frac{\dim V_{\la_j} \cdot \chi_{\la_j}(C)}{\dim \mf g} \\
 			&= \frac{|\nu|^2}{2} \frac{\tr(C,V)}{\dim \mf g}. \\
			\end{split}
			\eeq
 This must be even for all $\nu \in \ul \nu$; equivalently
 \beq
 p(\ul \nu) \cdot \frac{\tr(C,V)}{ \dim \mf g}
\eeq
must be even.
  \end{proof}

  \begin{cor}  \label{even.aspins} Let $\mf g$ be simple, and let $\varphi=\varphi_1 \oplus \varphi_2$ with $\varphi_1,\varphi_2$ orthogonal. Then $\varphi$ is spinorial iff either both $\varphi_1,\varphi_2$ are spinorial, or both $\varphi_1,\varphi_2$ are aspinorial.
  \end{cor}
  
  \bigskip
  
   The  following corollary will be useful when varying the isogeny class of $G$:
  
  \begin{cor} \label{ps.equal}  Let $\rho: \tilde G \to G$ be a cover, with simple Lie algebra, and let $\underline{\tilde{ \nu}}, \ul \nu$ be two sets of cocharacters, with $\underline{\tilde{ \nu}}$ generating $\pi_1(\tilde G)$, and $\ul \nu$ generating $\pi_1(G)$. Suppose that $\ord_2(p(\underline{\tilde{ \nu}}))=\ord_2(p(\ul \nu))$. Then an orthogonal representation $\varphi$ of $G$ is spinorial iff $\ol \varphi=\varphi \circ \rho$ is spinorial.
 \end{cor}
 
 \begin{proof} This follows since then
\beq
\begin{split}
\ord_2(q_\varphi )&=\ord_2(p(\ul \nu) \tau(\varphi) )\\
		&=\ord_2(p(\underline{\tilde{ \nu}})\tau(\ol \varphi) )\\
		&= \ord_2(q_{\ol \varphi}).\\
		\end{split}
\eeq
\end{proof}

 \subsection{A Counterexample}
 
The simplicity hypothesis for Corollary \ref{even.aspins} is necessary, for example let $G_1$ and $G_2$ be connected semisimple groups, with orthogonal representations $(\varphi_1,V_1)$ and $(\varphi_2,V_2)$ respectively.  
Let $G=G_1 \times G_2$, and write $\Phi_i: G_1 \times G_2 \to \SO(V_i)$ for the inflations of $\varphi_1,\varphi_2$ to $G$ via the two projections.  
Put $\Phi=\Phi_1 \oplus \Phi_2$.  For $\nu_1,\nu_2$ cocharacters of tori of $G_1,G_2$, put $\nu=\nu_1 \times \nu_2$.  It is easy to see that
\beq
L_\Phi(\nu)=L_{\varphi_1}(\nu_1)+L_{\varphi_2}(\nu_2).
\eeq
Therefore in this situation,
\beq
\begin{split}
\Phi \text{ is spinorial } & \Leftrightarrow \varphi_1 \text{ and } \varphi_2 \text{ are spinorial} \\
				& \Leftrightarrow 	\Phi_1 \text{ and } \Phi_2 \text{ are spinorial.} \\
				\end{split}
				\eeq

For example, if $G=\SO(3) \times \SO(3)$, and $\varphi_1,\varphi_2$ are aspinorial (e.g. the defining representation of $\SO(3)$), then each of $\Phi_1,\Phi_2$, and $\Phi_1 \oplus \Phi_2$ is aspinorial.
 \bigskip

 \section{Dynkin Index} \label{s7}

Let $\mf g$ be a simple Lie algebra with a long root $\alpha$. The quantity
 \beq
\check h=\frac{1}{ |\alpha|^2}  
 \eeq
is  called the \emph{dual Coxeter number of $G$}. (See Section 2 of \cite{Kostant}.)

 Following  \cite{Dynkin.Russian}, we define a bilinear form on $\mf t$ by
\beq
 (x,y)_d=2 \check h \cdot (x,y),
 \eeq
 for $x,y \in \mf g$.
 In other words, we renormalize the Killing form so that $(\alpha,\alpha)_d=2$. 

\begin{defn} Let $\phi: \mf g_1 \to \mf g_2$ be a homomorphism of simple Lie algebras.  Then there exists an integer $\dyn(\phi)$, called the Dynkin index of $\phi$,
so that for $x,y \in \mf g$, we have
\beq
(\phi(x),\phi(y))_d=\dyn(\phi) \cdot (x,y)_d.
\eeq
\end{defn}
 
 If $\phi \neq 0$, then $\dyn(\phi) \neq 0$.    Also, if $f': \mf g_2 \to \mf g_3$ is another homomorphism of simple Lie algebras, then $\dyn(f' \circ f)=\dyn(f')\dyn(f)$.  We refer the reader to  \cite{Dynkin.Selected},  page 195, Theorem 2.2, and (2.4).  

 We assume for the rest of this section that $\mf{so}_V$ is simple, equivalently $\dim V \neq 1,2,4$.  Note that there are no nontrivial irreducible orthogonal representations of $\mf g$ with those degrees.  The following is an easy calculation:
\begin{lemma} If $\iota_V: \mf{so}_V \hookrightarrow \mf{sl}_V$ is the standard inclusion, then $\dyn(\iota_V)=2$.
\end{lemma}
 \qedsymbol
 
Now let $\varphi: \mf g \to \mf{sl}_V$ be a nontrivial orthogonal Lie algebra  representation.  Then we may write $\varphi = \iota_V \circ \varphi'$,
where $\varphi': \mf g \to \mf{so}_V$.  We define $\dyn^o(\varphi)=\dyn(\varphi') \in \N$; thus $\dyn(\varphi)=2 \dyn^o(\varphi)$.

\begin{thm} \label{dynkin.thm}  For $\varphi: \mf g \to \mf{sl}_V$ a  representation, we have
 \beq
 \dyn(\varphi)= 2 \check h \dfrac{\tr(C; V)}{\dim \mf g}.
 \eeq
 \end{thm}
 
  \begin{proof} This is a reformulation of Theorem 2.5 of \cite{Dynkin.Selected}, page 197.
  \end{proof}
 
\begin{cor} \label{blum} Let $G$ have simple Lie algebra $\mf g$, and $\varphi$ an orthogonal representation of $G$. For a cocharacter $\nu$ we have
 \beq
 \begin{split}
q_{\varphi}(\nu)&=\frac{ |\nu|^2}{2} \cdot \frac{\dyn^o(\varphi)}{ \check h}.\\
 		\end{split}
 \eeq
 Therefore $\varphi$ is spinorial iff
 \beq
 p(\ul \nu) \cdot \frac{\dyn^o(\varphi)}{ \check h}
 \eeq
 is even.
  \end{cor}
 
 \bigskip
 
This formula is convenient because for simple $\mf g$, the dual Coxeter numbers are tabulated in Section 6 of  \cite{Kac}, and Dynkin indices for fundamental representations are found in Table 5 of  \cite{Dynkin.Russian}. In the forthcoming examples, we will use these tables without further comment.

 \section{Tensor Products} \label{tensor.section}
 
 In this section, we explain how the spinoriality of a tensor product of two orthogonal representations is related to the spinoriality of the factors.    \subsection{Internal Products}

Let $G$ be a connected reductive group and $(\varphi_1,V_1),(\varphi_2,V_2)$ orthogonal representations of $G$.  Write $(\varphi,V)=(\varphi_1 \otimes \varphi_2,V_1 \otimes V_2)$ for the (internal) tensor product representation of $G$.

 \begin{prop}  For $\nu  \in X_*(T) $, we have
\begin{equation} \label{internal.qs}
 q_\varphi(\nu)= \dim V_1 \cdot q_{\varphi_2}(\nu)+\dim V_2 \cdot q_{\varphi_1}(\nu).
 \end{equation}
 \end{prop}
 
 \begin{proof} 
 For $t \in F^\times$, we have
 \beq
 \Theta_{\varphi}(\nu(t))=\Theta_{\varphi_1}(\nu(t)) \Theta_{\varphi_2}(\nu(t)).
 \eeq
 Therefore
 \beq
  Q''_{(\varphi,\nu)} = Q_{(\varphi_1,\nu)}Q''_{(\varphi_2,\nu)} + 2 Q'_{(\varphi_1,\nu)}Q'_{(\varphi_2,\nu)} + Q''_{(\varphi_1,\nu)}Q_{(\varphi_2,\nu)},
  \eeq
 and so
 \beq
   Q''_{(\varphi,\nu)}(1) =   \dim V_1 \cdot Q''_{(\varphi_2,\nu)}(1)+ \dim V_2  \cdot Q''_{(\varphi_1,\nu)}(1). 
   \eeq
The proposition follows.
\end{proof}

  \begin{cor} \label{simple.tensor}  If $\varphi_1,\varphi_2$ are spinorial, then so is $\varphi_1 \otimes \varphi_2$.
 \end{cor}

 \subsection{External Tensor Products}
 
 Next, let $(\varphi_1,V_1),(\varphi_2,V_2)$ be orthogonal representations of connected reductive groups $G_1,G_2$, respectively.  Write $(\varphi,V)=(\varphi_1 \boxtimes \varphi_2,V_1 \otimes V_2)$ for the external tensor product representation of $G=G_1 \times G_2$.  If $T_1,T_2$ are maximal tori for $G_1,G_2$, then $T=T_1 \times T_2$ is a maximal torus of $G$.

 As in the previous proposition, we have: 
 \begin{prop} For $\nu=(\nu_1,\nu_2) \in X_*(T)=X_*(T_1) \oplus X_*(T_2)$, we have
 \beq
 q_\varphi(\nu)= \dim V_1 \cdot q_{\varphi_2}(\nu_2)+\dim V_2 \cdot q_{\varphi_1}(\nu_1).
 \eeq
 \end{prop}
  
 \subsection{Positive Orthogonal Spanning Sets}
 
 In the examples to come, it will be convenient to have a set of orthogonal dominant weights of $G$ which play the role of fundamental weights, but in $X_{\orth}^+$.  
  
   \begin{defn}
Let $S_o$ be a set of dominant orthogonal weights. We say that $S_o$ is a positive orthogonal spanning set (POSS) for $G$, provided every dominant orthogonal weight  can be written as a nonnegative integral combination of $S_o$.  
\end{defn}
   
 The strategy will be to deduce the spinoriality of an arbitrary $\varphi_\la$ from the spinoriality of the representations $\varphi_{\la_0}$ with  $\la_0 \in S_0$.  
   
\begin{lemma} \label{induct.spin.here} Let $G$ be semisimple and $\mu_0,\nu_0$ dominant weights. Put $\la_0=\mu_0+\nu_0$. Suppose that $\Phi=\varphi_{\mu_0} \otimes \varphi_{\nu_0}$ is spinorial, and that one of the following conditions holds:
\begin{enumerate}
\item $\varphi_\la$ is spinorial for any dominant orthogonal $\la \neq \la_0$ with $\la \prec \la_0$.
\item $\varphi_\la$ is spinorial for any dominant orthogonal $\la$ with $|\la|<|\la_0|$. 
\end{enumerate}
Then $\varphi_{\la_0}$ is spinorial.
\end{lemma}

\begin{proof} 
 By Proposition  \ref{Samelson.book}, $\Phi$ decomposes into a sum of irreducible orthogonal representations $\varphi_\la$ possibly together with an $S(\sigma)$ summand.

Let us see that each $\varphi_\la$ is spinorial, for $\la \neq \la_0$. If the first condition holds, this is clear by  (\cite{Bou.Lie.7-9}, page 132, Proposition 9 i).

Suppose the second condition holds. Each weight $\la$ of $\Phi$ decomposes into $\la=\mu_1+\nu_1$, with $\mu_1$ a weight of $\varphi_{\mu_0}$ and $\nu_1$ a weight of $\varphi_{\nu_0}$. Therefore $\mu_0-\mu_1$ and $\nu_0-\nu_1$ are positive. Moreover  $\varphi_{\la_0}$ itself occurs with multiplicity one.

The inner product of a dominant weight with a positive one is nonnegative, thus
\beq
\begin{split}
|\la|^2 & \leq (\la,\mu_0+\nu_0) \\
	& \leq   |\mu_0|^2+|\nu_0|^2 + (\mu_0,\nu_1)+(\mu_1,\nu_0) \\
&\leq |\la_0|^2. \\
		\end{split}
		\eeq
 
Moreover by (\cite{Bou.Lie.7-9}, page 129, Proposition 5(iii)),  equality holds iff $\mu_0=\mu_1$ and $\nu_0=\nu_1$, i.e., iff $\la=\la_0$. 
Thus by the second condition, each orthogonal weight $\la$ of $\Phi$, except a priori $\la_0$, has $\varphi_\la$ spinorial.  Recall that $S(\sigma)$ is spinorial by Proposition \ref{l.vs.s(nu)}.  So by Corollary \ref{even.aspins}, it must be that $\varphi_{\la_0}$ is spinorial.

\end{proof}

\begin{prop} \label{pos.basis} Let $\mf g$ be simple and suppose $S_o$ is a POSS for $G$.  If $\varphi_\la$ is spinorial for each $\la \in S_o$, then all orthogonal representations of $G$ are spinorial.
\end{prop}

\begin{proof} By Proposition \ref{l.vs.s(nu)} and Corollary \ref{even.aspins} we reduce to the case of irreducible orthogonal  $\varphi_{\la_0}$.  We prove the proposition by induction on $|\la_0|$.

If $\la_0 \in S_o$ then $\varphi_{\la_0}$ is spinorial. Otherwise $\la_0=\mu_0+\nu_0$ with $\nu_0 \in S_o$, and $\mu_0$ dominant orthogonal. Since
\beq| \la_0|^2=|\mu_0|^2+2(\mu_0,\nu_0)+|\nu_0|^2 > |\mu_0|^2,
\eeq
we can say that $\varphi_{\mu_0}$ is spinorial. Put $\Phi=\varphi_{\mu_0} \otimes \varphi_{\nu_0}$.

 By Corollary \ref{simple.tensor},  $\Phi$ is a spinorial orthogonal representation of $G$. 
 Therefore $\varphi_{\lambda_0}$ is spinorial, by  Lemma \ref{induct.spin.here}.
 \end{proof}

 \section{Type $A_{n-1}$} \label{a.section}
   
 For the next few sections, we will pursue the question:   For which groups $G$, with $\mf g$ simple, is every orthogonal representation spinorial? 
 This section treats the quotients of $\SL_n$.
  
  \subsection{Preliminaries for Type $A_{n-1}$}
   
Let $n$ be an even positive integer.   The center of $\SL_n$ is cyclic of order $n$,  and can be identified with the group $\mu_n$ of $n^{th}$ roots of unity in $F^\times$. Let $T_1$ be the diagonal torus of $\SL_n$. Let $\vt_i \in X^*(T_1)$ be the character of $T_1$ given by taking the $i^{th}$ diagonal entry. The roots of $T_1$ are of the form $\vt_i-\vt_j$ for $i \neq j$.  

 Let $d$ be a divisor of $n$, and $\mu_d$ the subgroup of $\mu_n$ of order $d$. In this section we consider the spinoriality of orthogonal representations of $G_{d}=\SL_n/\mu_d$. 
The maximal torus $T_d< G_d$ is the image of $T_1$ under this quotient.

Recall that generally $X_*(T_d)$ injects into $\mf t$ by $\nu \mapsto d \nu(1)$. When $T_n$ is the diagonal torus of $\PGL_n$, the injection $X_*(T_n) \hookrightarrow \mf t$  can be identified with the natural injection

\beq
\frac{\Z^n}{ \Z (1,1,\ldots, 1)}  \hookrightarrow \frac{F^n}{F (1,1,\ldots, 1)}.
\eeq

In these terms, each coroot lattice $Q(T_d) =X_*(T_1)$ is given by
\beq
\left\{ (x_1, \ldots, x_n) \in \frac{\Z^n}{ \Z (1,1,\ldots, 1)} \mid \sum_{i=1}^n x_i \equiv 0 \mod n \right\},
\eeq
and the subgroup $X_*(T_d) \leq X_*(T_n)$ is equal to
\beq
\left\{ (x_1, \ldots, x_n) \in \dfrac{\Z^n}{ \Z (1,1,\ldots, 1)} \mid \sum_{i=1}^n x_i \equiv 0 \mod \frac{n}{d} \right\} 
\eeq
   
    Let $\nu_d \in X_*(T_d)$ be the cocharacter  parametrized by $(\frac{n}{d},0, \ldots,0)$ in these terms. Then $\nu_d$ generates $\pi_1(G_d)$, which is therefore cyclic of order $d$. Of course, when $d$ is odd, every orthogonal representation of $G_d$ is spinorial. Let us henceforth take $d$ even. We compute
    \beq
    p(\nu_d)=\left( \frac{n}{d} \right)^2(n-1).
    \eeq

    Thus by Theorem \ref{mthm} we deduce:
  \begin{prop} \label{a.type.spin} 
 Let $\varphi_\la$ be an irreducible orthogonal   representation of $G_d$. Then $\varphi_\la$ is spinorial iff
 \beq
 \left( \frac{n}{d} \right)^2 \dim V_\la \cdot \chi_\la(C)
 \eeq
 is even.
 \end{prop}
 (We regard a rational number as \emph{even} if, when written in lowest terms, its numerator is even.)
 
 \begin{example}  
 Let $\varphi_\la$ be an irreducible orthogonal representation of $\GL_n$. By Proposition \ref{mod.z.spin}, it descends to an orthogonal representation $\ol \varphi_\la$ of $G_n=\PGL_n$, and $\varphi_\la$ is spinorial iff $\ol \varphi_\la$ is. By the above, we deduce that  $\varphi_\la$ is spinorial iff
  $\dim V_\la \cdot \chi_\la(C)$ is even.  
   \end{example}

Proposition \ref{a.type.spin} does not by itself answer the question at the beginning of this section, and the groups $G_d$ are somewhat awkward to compute with directly. So instead we ask, which morphisms from $\SL_n$ to $\Spin(V)$ descend to $G_d$?

\subsection{Descent Method}

Consider the following approach to determining the spinoriality of an orthogonal $\varphi: G/C  \to \SO(V)$, where $G$ is a simply connected and $C$ is central. Write $\hat \varphi: G \to \Spin(V)$ for the lift of $\varphi$. Then $\varphi$ is spinorial iff $C \leq \ker \hat \varphi$. In this paragraph we pursue this method for certain $C$; this approach will tremendously simplify the theory for the groups $G/C=G_d$ of type $A_{n-1}$.

Resetting notation, let   $\hat \varphi: G \to \Spin(V)$ be a morphism, with  $G$ connected semisimple. Put $\varphi=\rho \circ \hat \varphi: G \to \SO(V)$, 
and suppose that $\varphi$ is irreducible.  
 Let $C \leq Z(G)$, and suppose that $C \leq \ker \varphi$. Then $\varphi$ descends to an orthogonal representation $\ol \varphi$ of $G/C$, which is spinorial iff $C \leq \ker \hat \varphi$.

Let $d$ be a positive even integer, and  $\zeta_d \in F^\times$ a primitive $d$th root of unity.   
If $\nu \in X_*(T)$ with $\nu(\zeta_d) \in C$, then for all weights $\mu$ of $\varphi$, we have $d | \lip \mu,\nu \rip$.

\begin{prop} \label{criterion.2020}
Let $\nu: \mb G_m \to T$ be a cocharacter, so that $C$ is generated by $\nu(\zeta_d)$. 
The following are equivalent:
\begin{enumerate}
\item $\ol \varphi$ is spinorial.
\item $\hat \varphi(\nu(\zeta_d))=1$.
\item $2d$ divides $L_\varphi(\nu)$.
\end{enumerate}
\end{prop}
 
 \begin{proof}Choose  $\Sigma$ as in  the proof of Proposition \ref{L.phi.crit}. Note that $\hat \varphi_* \nu \in X_*(\tilde T_V)$ is a lift of $\varphi_* \nu \in X_*(T_V)$.  By Lemma \ref{blue.shirt}, we have  $\hat \varphi (\nu(\zeta_d))=1$ iff
 \beq
 2d \mid \lip \omega_\Sigma,\varphi_*\nu \rip, 
 \eeq
which, as in the proof of Proposition  \ref{L.phi.crit}, is equal to $L_\varphi(\nu)$.
 \end{proof}
 
\subsection{Application to $\SL_n$}
 
We return to $G=\SL_n$, with $n$ even. For $1 \leq i \leq n$, put $\varpi_i =\vt_1+ \cdots + \vt_i$.  Let $\nu_0 \in X_*(T)$ be the cocharacter defined by
 \beq
 \nu_0(t)=\diag(t,t,\ldots, t,t^{1-n}).
 \eeq
 
 Let $d$ be an even divisor of $n$. Then  $\nu_0(\zeta_d)$ generates $\mu_d <G$, so by Proposition \ref{criterion.2020},  the representation $(\ol \varphi_\la,V)$ of $G_d$ is spinorial iff $2d$ divides $L_\varphi(\nu_0)$.

 \begin{prop} The adjoint representation of $G_d$ is spinorial iff $\dfrac{n}{d}$ is even.
 \end{prop}
 \begin{proof}
  For $\varphi=\ad$, we have
  \beq
  \begin{split}
  L_{\ad}(\nu_0) &= \sum_{\alpha \in R: \lip \alpha,\nu_0 \rip>0} \lip \alpha,\nu_0 \rip \\
  			&=\sum_{i=1}^{n-1} \lip \vt_i-\vt_n,\nu_0 \rip \\
			&=   n(n-1). \\
			\end{split}
			\eeq

   This is divisible by $2d$ iff $\dfrac{n}{d}$ is even.
  \end{proof}

\subsection{Case where $n/d$ is even}
 
For all   $q \in Q(T)$, the quantity $\lip q,\nu_0 \rip$ is divisible by $n$. Since all weights of $V_\la$ are congruent mod $Q(T)$, we deduce that 
\begin{equation} \label{chill}
L_{\varphi_\la}(\nu_0) \equiv  \lip \la,\nu_0 \rip \cdot \left(\sum_{\mu : \lip \mu,\nu_0 \rip >0} m_\mu \right) \mod n.
\end{equation}

 \begin{prop} \label{sl_n.done} Suppose that $2d$ divides $n$, and an irreducible orthogonal representation $\varphi_\la$ of $\SL_n$ descends to the orthogonal representation $\ol{\varphi_\la}$ of $G_d$.
 
 \begin{enumerate}
\item  If $V_\la$ is odd-dimensional, then $\ol \varphi_\la$ is spinorial.
 \item If $V_\la$ is even-dimensional, then $\ol \varphi_\la$ is spinorial iff the product $\half \dim V_\la \cdot \lip \la,\nu_0 \rip$ is divisible by $2d$.
\end{enumerate}
 
 \end{prop}
 
 \begin{proof} 
 
 If $\varphi_\la$ is orthogonal with odd degree, then the trivial weight must occur in $V_\la$, which implies that $\la \in Q(T)$. From \eqref{chill}, we see that $L_\varphi(\nu_0)$ is divisible by $n$, and the first statement follows.
 
 Now suppose $\varphi_\la$ has even degree. If $\lip \la,\nu_0 \rip$ is divisible by $n$, then $\ol \varphi_\la$ is spinorial and the second statement is clear.
 If $\lip \la,\nu_0 \rip$ is not divisible by $n$, then for all $\mu$ occurring in $V_\la$, it must be that $\lip \mu, \nu_0 \rip \neq 0$. It follows that
 \beq
\sum_{\mu : \lip \mu,\nu_0 \rip >0} m_\mu = \half  \dim V_\la,
 \eeq
 and the second statement follows from \eqref{chill}.
 \end{proof}

 Let $\varpi_i^o=\varpi_i+\varpi_{n-i}$ for $1 \leq i < \frac{n}{2}$ and
 \beq
 S_o= \left\{\varpi_i^o \mid 1 \leq i<\frac{n}{2} \right\} \cup \left\{ \varpi_{n/2} \right\}.
 \eeq
 
It is easy to see that  $S_o$ is a POSS for $G_d$. Note that $\lip \varpi_i^o,\nu_0 \rip=n$ for $1 \leq i < \frac{n}{2}$, and $\lip \varpi_{n/2}, \nu_0 \rip=n/2$.

\begin{prop} Suppose $2d$ divides $n$. For each $ 1 \leq i<\frac{n}{2}$, the representation $V_{\varpi_i^o}$ of $G_d$ is spinorial.
\end{prop}

\begin{proof} This follows from Proposition \ref{sl_n.done}.
\end{proof}
 
 \bigskip
 
  The representation $V_\la$ for $\la=\varpi_{n/2}$ is the exterior power $\wedge^{n/2}V_0$, where $V_0$ is the standard representation of $\SL_n$.

 \begin{prop} Let $\la=\varpi_{n/2}$. Then $\ol \varphi_\la$ is aspinorial iff $n$ is a power of $2$ and $d=n/2$.
 \end{prop}
 
 \begin{proof} From elementary number theory we know that $\dim V_\la=\binom{n}{n/2}$ is even, and divisible by $4$ iff $n$ is not a power of $2$.
 We have
 \begin{equation} \label{binoms}
 \half \dim V_\la \cdot \lip \la,\nu_0 \rip=\frac{n}{4} \binom{n}{n/2}.
 \end{equation}
 This is divisible by $d$, since the binomial coefficient is always even, and $d$ divides $\frac{n}{2}$. Thus \eqref{binoms} is divisible by $2d$ unless both $\frac{n}{2d}$ is odd, and $n$ is a power of $2$. In this case it must be that $n$ is a power of $2$ and $n=2d$.
 \end{proof}

 \begin{thm} Suppose that $n/d$ is even. Unless $n=2^{k+1}$ for some $k \geq 1$ and $d=2^k$, every orthogonal representation of $G_d=\SL_n/\mu_d$ is spinorial.
 \end{thm}
 
  \begin{proof} 
  By Proposition \ref{pos.basis}, it is enough to check that $\ol \varphi_\la$ is spinorial for each $\la \in S_o$. But we have done this.
  \end{proof}

\begin{example}
Although the adjoint representation of $G_2=\SL_4/\{ \pm 1\}$ is spinorial, the representation $V_{\varpi_2}=\wedge^2 F^4$ of $G_2$ is aspinorial. 
\end{example} 
 
 \begin{remark} What makes the ``descent method" work in the case of $G_d$ with $n/d$ even is the fortunate fact that $\lip q,\nu_0 \rip$ is divisible by $n$ for $q \in Q(T)$. In other contexts, it is unclear how to compute $L_\varphi(\nu) \mod 2d$.
 \end{remark}

  \subsection{Summary for the groups $G_d$}
  
  Let $n$ be a positive integer, $d$ a divisor of $n$, and $G_d=\SL_n/\mu_d$.  
  
  From the above we have:
  
  \begin{itemize}
  \item If $d$ is odd, then every orthogonal representation of $G_d$ is spinorial.
  \item If $n$ is even and $n/d$ is odd, then the adjoint representation of $G_d$ is aspinorial.
  \item If $n$ is a power of $2$ and $d=n/2$, then $\wedge^{d} V_0$ is an aspinorial representation of $G_d$.
  \item If $n/d$ is even, then every orthogonal representation of $G_d$ is spinorial, unless $n$ is a power of $2$ and $d=n/2$.
  \end{itemize}
 
 In particular, every orthogonal representation of $G_d$ is spinorial iff $n$ is odd, or $n/d$ is even with $(n,d) \neq (2^{k+1},2^k)$.

\section{Type $C_n$} \label{c.section}

Let $J$ be the $2n \times 2n$ matrix $\mat 0I{-I}0$, where $I$ is the $n \times n$ identity matrix.  We let 
\beq 
G_{\simp}=\Sp_{2n} =\{g \in \GL_{2n} \mid g^t Jg=J\}.
\eeq
Write $T_{\simp}$ for the diagonal torus in $G_{\simp}$. A typical element is $\diag(t_1, \ldots, t_n, t_1^{-1}, \ldots, t_n^{-1})$.  
We identify $X_*(T_{\simp})$ with $\Z^n$ by $(b_1, \ldots, b_n) \mapsto \nu$, where
\beq
\nu(t)=\diag(t^{b_1}, \ldots, t^{b_n}, \ldots).
\eeq

Put $G=\Sp_{2n}/\{ \pm 1\}$, and let $T$ be the image of $T_{\simp}$ under the quotient. Then $X_*(T_{\simp})$ has index $2$ in $X_*(T)$; more precisely we may write
\beq
X_*(T)=X_*(T_{\simp})+\Z \cdot \nu_0,
\eeq
where
$\nu_0=\half(1,1,\ldots, 1)$. In particular, $\pi_1(G)$ is cyclic of order $2$, generated by $\nu_0$.

\begin{remark} One way to understand $\nu_0$ is through the isomorphism $\Sp_{2n}/\{\pm 1\} \cong \GSp_{2n}/Z$, where $\GSp_{2n}$ is the general symplectic group defined with $J$, and $Z$ is its center. The cocharacter 
\beq
t \mapsto \diag(\underbrace{t,t, \ldots, t}_{n \text{ times}}, 1,1, \ldots, 1),
\eeq
of the diagonal torus of $\GSp_{2n}$, when projected to $T$, is $\nu_0$.
\end{remark}

 We have $p(\nu_0)=\half n(n+1)$.
Every representation $\varphi$ of $G$ is orthogonal. Since $\check h=n+1$, we have by Corollary \ref{blum}:
 \begin{equation} \label{cn.form}
\begin{split}
 q_\varphi(\nu_0) &= \frac{\half n(n+1)}{\check h} \dyn^o(\varphi) \\
 		&=\half n \cdot \dyn^o(\varphi). \\
		\end{split}
		\end{equation}

 \begin{prop}  Every representation of $\Sp_{2n}/\{ \pm 1\}$ is spinorial iff $4|n$.
 \end{prop}
 
\begin{proof} If $n \equiv 0 \mod 4$, then every representation is spinorial by \eqref{cn.form}. If $n \equiv 1,2 \mod 4$, then the adjoint representation is aspinorial, and if $n \equiv 3 \mod 4$, then 
the second fundamental representation is aspinorial.
\end{proof}

 \section{Type $D_n$} \label{d.section}
 
 The simply connected group of type $D_n$ is $G_{\simp}=\Spin_{2n}$.  The center $Z$ of $G_{\simp}$ has order $4$; in the notation of  Section \ref{tori.spin.section}, it is generated by $c^+$ when $n$ is odd, and generated by $c^+$ and $z$ when $n$ is even.

Thus the groups of type $D_n$ for $n$ odd are $G_{\simp}$ and its quotients $\SO_{2n}$ and $\PSO_{2n}$.  When $n$ is odd, the adjoint representation of $\PSO_{2n}$ is aspinorial, which ends our investigation in this case. Henceforth in this section, we will assume that $n$ is even, and to ensure $\mf g$ is simple we take $n>2$. (See Example \ref{so.4.example} for $\SO_4$.)

For $n$  even, there are two more groups of type $D_{2n}$: the quotient  $G^+_{2n}$  of $G_{\simp}$ by $\lip c^+ \rip$, and the quotient $G^-_{2n}$ of $G_{\simp}$ by $\lip -c^+ \rip$. Write 
$T_{2n}=T_V$ where $V=F^{2n}$, write $\widetilde T_{2n}< G_{\simp}$ for its preimage, and $T_{2n}^{\pm}$, and $\ol T_{2n}$ for the corresponding tori of $G_{2n}^{\pm}$ and $\PSO_{2n}$. The lattice of cocharacters corresponding to these quotients is depicted in Figure 1.

 \begin{figure}
\centering
\begin{tikzpicture}[node distance=2.5cm,line width=1pt]
\title{Cocharacter Lattice}
\node(PT) at (0,0)     {$X_*(\overline T_{2n})$};
\node(T+) [below left=0.5cm of PT] {$X_*(T_{2n}^+)$};
\node(T) [below =0.4cm of PT] {$X_*(T_{2n})$};
\node(T-) [below right=0.5cm of PT] {$X_*(T_{2n}^-)$};
\node(TSC) [below=0.5cm of T] {$X_*(\widetilde T_{2n})$};

\draw(PT) --(T+);
\draw(PT)-- (T);
\draw(PT) -- (T-);
\draw(T) --(TSC);
  \draw(T+) --(TSC);
   \draw(T-) --(TSC); 
 
\end{tikzpicture}
\caption{Cocharacter Lattice for $D_{2n}$}
\end{figure}
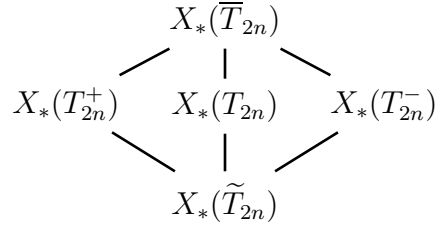

Recall that we identified $X_*(T_{2n})$ with $\Z^n$ in Section \ref{tori.spin.section}.
Let
\beq
Q=X_*(\widetilde T_{2n})=\left \{ (b_1, \ldots, b_n) \in  \Z^n \mid \sum b_i \text{ is even} \right\}.
\eeq
Then
\beq
X_*(T_{2n}^{\pm})=Q+\Z \cdot \half(1,1, \ldots, \pm 1)
\eeq
and
\beq 
X_*(\ol T_{2n})=\Z^n+\Z \cdot \half(1,1, \ldots,  1).
\eeq 
 All representations of groups of type $D_n$ are orthogonal.
The standard representation $V_0$ of $\SO_{2n}$ is evidently aspinorial.   Thus, for the rest of this section we   focus on the groups $\PSO_{2n}$ and $G^{\pm}_{2n}$ with $n$ even.

Tables 1 and 2 below record the quantities $p(\ul \nu)$, $\dim V_\la$, and $\chi_\la(C)$, that we need for our formulas.
Here $\varpi_k=(\underbrace{1,1,\ldots,1}_{k \text{ times}},0,\ldots, 0)$ and $\varpi_-=(1,1,\ldots, 1, -1)$. (We parametrize $X_*,X^*$ by $\Z^n$ as in Section \ref{tori.spin.section}

\begin{table}[ht]
	\caption{Computing $p(\ul \nu)$}
	\centering	

\tabcolsep=.5cm
 
\begin{tabular}{|cc|c|c|}
 \hline
 
 $G$  & &    $\ul \nu$ &  $p(\ul \nu)$    \\
\hline
\hline
$\SO_{2n}$ & & $(1,0,\ldots,0)$& $2n-2$ \\
\hline

 $\PSO_{2n}$   & $n \equiv 0 \mod 4$  &  $(1,0,\ldots,0)$, $\half(1,\ldots,1)$ &   $2n-2$\\  
 \hline
 $\PSO_{2n}$ & $n \equiv 2 \mod 4$  &  $(1,0,\ldots,0)$, $\half(1,\ldots,1)$ & $n-1$  \\
\hline
$G^{\pm}_{2n}$ &   & $ \half(1,\ldots,\pm 1)$ &$\dbinom{n}{2}$ \\
 
\hline

\end{tabular}
\label{exam3}
\end{table}

\begin{table}[ht]
	\caption{$\dim V_\la$ and $\chi_\la(C)$ for type $D_n$, $n$ even}
	\centering	

\tabcolsep=.5cm
\begin{tabular}{|c|c|c|}
 \hline
 
 $\la$  & $\dim V_\la$&  $\chi_\la(C)$    \\
\hline
$\varpi_k$ & $\dbinom{2n}{k}$ & $\dfrac{k(2n-k)}{4n-4}$ \\
\hline
$\half \varpi_-$ & $2^{n-1}$ & $\dfrac{n(2n-1)}{16(n-1)}$ \\
 \hline
 $\half \varpi_n$ & $2^{n-1}$ & $\dfrac{n(2n-1)}{16(n-1)}$\\
 \hline
 $\varpi_-$ & $\dfrac{(2n-1)!}{2^n}$ & $\dfrac{n^2}{4n-4}$ \\
 \hline

\end{tabular}
\label{exam}
\end{table}

\begin{remark} Let $n$ be odd. One has similarly $p(\nu_0)=n-2$ for $\SO_n$, with 
\beq
\nu_0(t)=\diag(t,1, \ldots, 1,t^{-1}).
\eeq
\end{remark}

 \subsection{The Case of $\PSO_{2n}$}
 
  If $n \equiv 2 \mod 4$, then the  representation $\varphi$ of   $\PSO_{2n}$ on  $\wedge^2 V_0$  is aspinorial by Corollary \ref{blum}:  here $\dyn^o(\varphi)=2n-2$ and $\check h=2n-2$, so 
  \beq
  p(\ul \nu) \cdot \frac{\dyn^o(\varphi)}{\check h}=n-1.
  \eeq
Let us assume for the rest of this section that $n$ is a multiple of $4$;  we will prove every orthogonal representation is spinorial in this case. We first consider $\varphi_\la$, with $\la$ in the set
\beq
 S_0=\{ \varpi_k, \varpi_- \mid k \text{ even}, 1 \leq k \leq n \}.
 \eeq

 \begin{prop} \label{s0.spinorial} Each representation $\varphi_\la$ of $\SO_{2n}$ with $\la \in S_0$ is spinorial. 
 \end{prop}
 
 \begin{proof} 
  Tables 1 and 2 give
 \begin{equation} \label{flicker1}
 \begin{split}
 q_{\varpi_k} &=\frac{1}{2} \frac{\binom{2n}{k}}{\binom{2n}{2}} \cdot k(2n-k) \\
 		&=\binom{2n-2}{k-1}.\\
		\end{split}
 \end{equation}
 
 Since $k$ is even, this is necessarily even, and we deduce that each $\varphi_{\varpi_k}$ is spinorial. Similarly  
 \begin{equation} \label{flicker2}
 q_{\varpi_-}=\frac{(2n-2)!n}{2^{n+1}};
 \end{equation}
 it is easy to see this is even for all $n$ divisible by $4$, thus $\varphi_{\varpi_-}$ is spinorial. \end{proof}
 
 These representations descend to $\PSO_{2n}$, which are also spinorial by Corollary \ref{ps.equal}. Moreover, formulas \eqref{flicker1} and  \eqref{flicker2}  remain the same when computed for  $\PSO_{2n}$ (since the $p(\ul \nu)$ are the same).

 Let $S_1$ be the set
 \beq
 S_0 \cup \{ \varpi_k+\varpi_\ell \mid k \equiv \ell \mod 2, 1 \leq k,\ell \leq n \} \cup \{\varpi_k+\varpi_- \mid k \text{ even} \} \cup \{2 \varpi_- \}.
 \eeq
 
 Note that $S_1$ has the following property: If $\la \in S_1$ and $\la'$ is a dominant weight with $\la' \prec \la$, then $\la' \in S_1$.

 \begin{prop} \label{s1.ad} Each $\varphi_\la$ with $\la \in S_1$ is spinorial.
 \end{prop}
 
 \begin{proof} 
 Suppose, by way of contradiction, that there are aspinorial $\la \in S_1$. Let $\la \in S_1$ be a $\prec$-minimal  with $\varphi_\la$ aspinorial.
 We show below that for each of the three possibilities of $\la \in S_1$, the representation $\varphi_\la$ is spinorial, a contradiction.
 
  If $\la \in S_0$  it is spinorial by Proposition \ref{s0.spinorial}.
 Otherwise $\la=\la_1+\la_2$ with $\la_1,\la_2 \in S_0$, or $\la=\varpi_k+\varpi_\ell$ with $k,\ell$ odd.
 
   In the first case, let $\Phi=\varphi_{\la_1} \otimes \varphi_{\la_2}$, which is spinorial by Proposition \ref{s0.spinorial}.
By the property of $S_1$ mentioned above, we may apply Lemma \ref{induct.spin.here} (1) to deduce that $\varphi_{\la}$ is spinorial.
 
 In the second case we have $\la=\varpi_k+\varpi_\ell$ with $k,\ell$ odd. Consider the representation $\Phi=\varphi_{\varpi_k} \otimes \varphi_{\varpi_\ell}$ of $\SO_{2n}$. Applying Equation \eqref{internal.qs} to the representations $\varphi_{\varpi_k}$ and $\varphi_{\varpi_\ell}$ of $\SO_{2n}$ gives 
\beq
q_{\Phi}=\binom{2n}{\ell} \binom{2n-2}{k-1}+ \binom{2n}{k} \binom{2n-2}{\ell-1}.
\eeq
Since this is even,  $\Phi$ is a spinorial representation of $\SO_{2n}$. By Corollary \ref{ps.equal} it descends to  a spinorial representation  $\ol \Phi$  of $\PSO_{2n}$.
 
Again, we may apply Lemma \ref{induct.spin.here} (1) to deduce that $\varphi_{\la}$ is spinorial. 
 Thus in all cases we have a contradiction. \end{proof}
 
 \begin{thm} \label{pso.all} When $n$ is divisible by $4$, every representation of $\PSO_{2n}$ is spinorial.
 \end{thm}
 
 \begin{proof} It is elementary to see that  $S_1$ is a POSS. 
 Thus the theorem follows by Propositions  \ref{pos.basis} and \ref{s1.ad}.
 \end{proof}
 
  \subsection{The groups $G^{\pm}_{2n}$}
  
Here $\check h=2n-2$ and $p(\nu_0)=\binom{n}{2}$, so by Corollary \ref{blum}:
 \begin{equation} \label{dn.form}
\begin{split}
 q_\varphi(\nu_0) &= \binom{n}{2} \cdot \frac{\dyn^o(\varphi)}{\check h} \\
 		&=\frac{n}{4}\cdot \dyn^o(\varphi). \\
		\end{split}
		\end{equation}

 If $n \equiv 2 \mod 4$, then the  representation $\varphi$ of   $G^{\pm}_{2n}$ on  $\wedge^2 V_0$  is aspinorial, since again $\dyn^o(\varphi)=2n-2$. The half-spin representation $(\varphi_{\half \varpi_4}, V_{\half \varpi_4})$ of $G^+_8$, and the half-spin representation $(\varphi_{\half \varpi_-},V_{\half \varpi_-})$ of $G^-_8$ are also aspinorial, since
  here $\dyn^o(\varphi)=1$.

\begin{thm} Suppose $n>4$ and a multiple of $4$.  Then every orthogonal representation of $G^+_{2n}$ and $G^-_{2n}$ is spinorial.
\end{thm}

\begin{proof}
If $n$ is a multiple of $8$, then the conclusion follows from \eqref{dn.form}.

If $n \equiv 4 \mod 8$, then $\ord_2(p(\ul \nu))=1=\ord_2(p(\ul \nu'))$. Therefore we may apply Corollary \ref{ps.equal} to see that a representation of $G^{\pm}_{2n}$ which descends to $\PSO_{2n}$ is spinorial iff it was originally spinorial. Thus by Theorem \ref{pso.all}, all such representations of $G^{\pm}_{2n}$ are spinorial.

However there are   representations of $G^{\pm}_{2n}$ which don't descend, so we must enlarge our POSS.  Let $S^+=S_1 \cup \{\half \varpi_n\}$ and $S^-=S_1 \cup \{ \half \varpi_-\}$. Then $S^{\pm}$ is a POSS for $G^{\pm}$.  By \eqref{dn.form}, we have
 
\beq
q_{\half \varpi_n} =q_{\half \varpi_-}=n 2^{n-6}, 
\eeq
which is certainly even.  Thus for each $\la \in S^{\pm}$, the representation $V_\la$ of $G^{\pm}_{2n}$ is spinorial. The conclusion then follows by Proposition \ref{pos.basis}. \end{proof}

    \subsection{Summary for groups of type $D_n$}
  
  Let $n>2$ be a positive integer. From the above we know:
  
  \begin{itemize}
  \item The standard representation of $\SO_{2n}$ is aspinorial.
  \item If $n$ is a multiple of $4$, then every representation of $\PSO_{2n}$ is spinorial.
  \item If $n$ is odd, then the adjoint representation of $\PSO_{2n}$ is aspinorial.
  \item If $n \equiv 2 \mod 4$, then the  representations of $\PSO_{2n}$ and $G^{\pm}_{2n}$ on  $\wedge^2 V_0$  are aspinorial.
 \item The half-spin representation $\varphi_{\half \varpi_4}$ of $G^+_8$, and the half-spin representation $\varphi_{\half \varpi_-}$ of $G^-_8$ are aspinorial.
 \item For $n>4$ a multiple of $4$, all representations of $G^+_{2n}$ and $G^-_{2n}$ are spinorial.
  \end{itemize}
 
 \section{Summary for Simple $\mf g$} \label{summary.section}

Here is a list of all $G$ with simple $\mf g$, with the property that all orthogonal representations of $G$ are spinorial:

\begin{itemize}
\item All $G$ whose fundamental group has odd order
\item All $\SL_n/\mu_d$, when $n/d$ is even, except when $n$ is a power of $2$ and $d=n/2$
\item $\Sp_{2n}/{\pm 1}$, when $n$ is a multiple of $4$
\item The groups $\PSO_n$, when $n$ is a multiple of $8$
\item The groups $G^{\pm}_{2n}$, when $n>4$ is a multiple of $4$
\end{itemize}

   For the reader's convenience, we recall  the $G$ whose fundamental groups have odd order:
   
   \begin{itemize}
   \item Simply connected $G$
   \item $\SL_n/\mu_d$ with $d$ odd
\item The adjoint group of type $E_6$  
\end{itemize}

 Aspinorial representations for most groups not on this list have already been mentioned. To finish, we remark that the standard representation of an odd orthogonal group is aspinorial, and the adjoint representation of the adjoint group of type $E_7$ is aspinorial.

\section{Periodicity}   \label{s10}

For the irreducible orthogonal  representations $\varphi_\la$, our lifting criterion amounts to determining the parity of one or more $q_\la(\nu)$, each an integer-valued polynomial function of $\la$.  As we explain in this section, this entails a certain periodicity of the spinorial highest weights in the character lattice. 

\subsection{Polynomials with Integer Values}

Let $V$ be a finite-dimensional rational vector space, $V^*$ its dual, $L$ a lattice in $V$, and $L^\vee$ the dual lattice in $V^*$.  Recall that $L^\vee$ is the $\Z$-module of $\Q$-linear maps $f: V \to \Q$ so that $f(L) \subseteq \Z$.
Denote by $\dbinom{L^\vee}{\Z}$ the $\Z$-algebra of polynomial functions on $V$ which take integer values on $L$.  Given $f \in L^\vee$, and $n \in \N$, define $\dbinom{f}{n}  \in \dbinom{L^\vee}{\Z}$ by the prescription
\beq
\binom{f}{n}: x \mapsto \binom{f(x)}{n}=\frac{f(x)(f(x)-1) \cdots (f(x)-n+1)}{n!}
\eeq
for $x \in L$.   

\begin{prop} \label{bourb.lattice} The $\Z$-algebra $\dbinom{L^\vee}{\Z}$ is generated by the $\dbinom{f}{n}$ for $f \in L^\vee$ and $n \in \N$.  If $\{f_1, \ldots, f_r\}$ is a $\Z$-basis of $L^\vee$, then the products
\beq
\binom{f_1}{n_1} \cdots \binom{f_r}{n_r},
\eeq
where $n_1, \ldots, n_r \in \N$, form a basis of the $\Z$-module $\dbinom{L^\vee}{\Z}$.
\end{prop}

\begin{proof} See Proposition 2 in  \cite{Bou.Lie.7-9}, Chapter 8, Section 12, no. 4.
\end{proof} 

Given a basis of $V$, we can form the set $C$ of its nonnegative linear combinations.  Call $C$ a ``full polyhedral cone" if it arises in this way, and write
 $L^+=L \cap C$.
\begin{prop} \label{rohitsproof} Suppose $f$ is a polynomial map from $V$ to $\Q$ that take integer values on $L^+$.  Then $f \in \binom{L^\vee}{\Z}$.
\end{prop}

\begin{proof}

We omit the elementary proof (see \cite{Rohit.thesis}) of  the following lemma:

\begin{lemma} Suppose that $V$ is a finite-dimensional rational vector space, that $C$ is a full polyhedral cone in $V$, and that $L \subset V$ is a lattice.  Let $p \in L$.  Then 
\begin{enumerate}
\item $C \cap (p+C)$ is a translation of $C$.
\item The intersection $L \cap C \cap (p+C)$ is nonempty.
\item Suppose $p'$ is in the above intersection, and write $v=p'-p$.  Then $p+nv \in L \cap C \cap (p+C)$ for all positive integers $n$.
\end{enumerate}
\end{lemma}

Continuing with the proof of the proposition, let $\ell \in L$; we must show that $f(\ell) \in \Z$.  By the lemma there is a $v \in L$ so that $\ell+nv \in L^+$ for all positive integers $n$.
For $x \in \Z$, put $g(x)=f(\ell+x v)$.  Then $g \in \Q[x]$, and by hypothesis it takes integer values on positive integers.  It is elementary to see that such a polynomial 
takes integer values at {\it all} integers, and in particular $g(0)=f(\ell) \in \Z$. \end{proof}

\begin{lemma} \label{binom.period} Fix an integer $n \geq 1$ and put $k=\left[ \log_2n \right]+1$.  Then $\dbinom{a+2^k}{n} \equiv \dbinom{a}{n} \mod 2$ for every integer $a \geq 1$.
\end{lemma}
 
 \begin{proof} This follows from the Lucas Congruence (see e.g.,\cite{ec1}).
 \end{proof}
 
 \begin{prop} \label{2.power.period} Let $f \in \dbinom{L^\vee}{\Z}$.  Then there is a $k \in \N$ so that for all $x,y \in L$ we have
 \beq
 f(x + 2^ky) \equiv f(x) \mod 2.
 \eeq
 \end{prop}
 
 \begin{proof} By Proposition \ref{bourb.lattice}, there are $f_1, \ldots, f_r \in L^\vee$, integers $n_1, \ldots, n_r$, and a polynomial $g \in \Z[x_1, \ldots, x_r]$ so that
 \beq
 f=g \left(\binom{f_1}{n_1}, \ldots, \binom{f_r}{n_r} \right).
 \eeq
Let $k_i=\left[ \log_2n_i \right]+1$;  by Lemma \ref{binom.period} we have
 \beq
\binom{f_i}{n_i}(x+2^{k_i}y) \equiv \binom{f_i}{n_i}(x) \mod 2
\eeq
for all $x,y \in L$.  If we put     $k=\max(k_1, \ldots, k_r)$ we obtain the proposition.
 \end{proof}

\subsection{Example:  Parity of Dimensions}
 
To illustrate the above, let $G$ be connected reductive with notation as before.  Take $L=X^*(T) \subset V=X^*(T) \otimes \Q \hookrightarrow \mf t^*$.
Define $f: \mf t^* \to F$ by
\beq
f(\la)=\frac{d_{\la+\delta}}{d_\delta}=\dim V_\la.
\eeq
From Propositions \ref{rohitsproof} and \ref{2.power.period}  we deduce:
\begin{cor} \label{earlier...} With notation as above:
\begin{enumerate}
\item $f(\la) \in \Z$ for all $\la \in X^*(T)$; equivalently $f \in \dbinom{X_*(T)}{\Z}$.
\item There is a $k \in \N$ so that $f(\la_0+2^k \la) \equiv f(\la_0) \mod 2$ for all $\la_0,\la \in X^*(T)$.
\end{enumerate}
\end{cor}

\subsection{Proof of Theorem \ref{period.intro}}  
 
We continue with $G$ connected reductive.

If $\mf g$ is simple put
\beq
\eta_{\ul \nu}(\la)=p(\ul \nu) \cdot \frac{\dim V_\la \cdot \chi_\la(C)}{\dim \mf g}.
\eeq
 Then:
\begin{enumerate}
\item $\eta_{\ul \nu}$ is a polynomial in $\la$,
\item $\eta_{\ul \nu}(\la)\in \Z$ for $\la \in X_{\sd}^+$, and
\item $\varphi_\la$ is spinorial iff $\eta_{\ul \nu}(\la)$ is even.
\end{enumerate}

If $\mf g$ is not necessarily simple, we may instead put 
\beq 
\eta_{\ul \nu}(\la) = 1+ \prod_{\nu \in \ul \nu}(q_\la(\nu)-1),
\eeq   
  and the same three properties hold.  From Propositions \ref{rohitsproof} and \ref{2.power.period}  we deduce:
  
\begin{cor} With notation as above,
\begin{enumerate}
\item $\eta_{\ul \nu}(\la) \in \Z$ for all $\la \in X_{\orth}$; equivalently $\eta_{\ul \nu} \in \dbinom{X_{\orth}^\vee}{\Z}$.
\item There is a $k \in \N$ so that $\eta_{\ul \nu}(\la_0+2^k \la) \equiv \eta_{\ul \nu}(\la_0) \mod 2$ for all $\la_0,\la \in X_{\orth}$.
\end{enumerate}
\end{cor}

Theorem \ref{period.intro} in the introduction follows from this.  If we put $L^+=2^k X_{\orth}^+$, then the theorem says that the set of spinorial highest weights is stable under  addition from $L^+$.  
Since the index $[X_{\orth}: 2^k X_{\orth}]$ is finite, the determination of the full set of spinorial weights amounts to a finite computation.

The problem of finding the exact largest lattice $L \subseteq X_{\orth}$ so that the spinorialities of $\varphi_{\la_0}$ and $\varphi_{\la_0+ \ell}$ agree for all $\la \in X_{\orth}^+$ and $\ell \in L^+$ seems interesting, as does the problem of determining the proportion of spinorial irreducible representations.  We do not settle these questions here, but see the next section for $\PGL_2$ and $\SO_4$, and \cite{Rohit.thesis} for more examples. 
   
 \subsection{Examples} 
 
  Let us examine $G=\PGL_2$ more closely.  We have $X^*(T)=X_{\sd}=X_{\orth}$.   For integers $j \geq 0$ define $\lambda_j \in X^*(T)$ by
\beq
\lambda_j\left( \mat a{} {} b \right)=(ab^{-1})^j.
\eeq
Then $\dim V_{\la_j}=2j+1$ and $\chi_{\la_j}(C)=\half (j^2+j)$, so $\varphi_{\lambda_j}$ is spinorial iff 
\beq
\frac{j(j+1)(2j+1)}{2}
\eeq
is even.  Equivalently, $j \equiv 0,3 \mod 4$.  We may therefore take $k=2$ in Theorem \ref{period.intro}.

\bigskip 

As a second example, recall the representations $V_{a,b}$ of $\SO_4$ from Example \ref{so.4.example}.
If we put 
 \beq
 F(a,b)=\frac{1}{4} \left(  (b+1) \binom{a+2}{3}+(a+1) \binom{b+2}{3} \right),
 \eeq
then $V_{a,b}$ is spinorial iff $F(a,b)$ is even.   It is elementary to see that $F(a+8i,b+8j) \equiv F(a,b) \mod 2$ for integers $i,j$.  In particular we may take $k=3$ in Theorem   \ref{period.intro}.

\section{Reduction to Algebraically Closed Fields} \label{reduction.closed}

For this section, $G$ is a connected reductive group defined over a field $F$ of characteristic $0$, not necessarily algebraically closed.  
Let $V$ be a quadratic vector space over $F$, and $\varphi: G \to \SO(V)$ a morphism defined over $F$.  The isogeny $\rho: \Spin(V) \to \SO(V)$ is also defined over $F$.  By extending scalars to the algebraic closure $\ol F$ of $F$, we may use the rest of this paper to determine whether there exists a lift 
$\hat \varphi: G \to \Spin(V)$ of $\varphi$ defined over $\ol F$.  

\begin{lemma}   If $\hat \varphi: G \to \Spin(V)$ is a lift defined over $\ol F$, then it arises from a lift defined over $F$.
\end{lemma}

\begin{proof} 
The Galois group acts by Zariski-continuous automorphisms on the $\ol F$-points of $G$ and $\Spin(V)$.   
We must show that for every $\sigma \in \Gal(F)$ and $x \in G(\ol F)$, we have ${}^\sigma \hat \varphi(x)=\hat \varphi({}^\sigma x)$.
Since $\rho$ and $\varphi$ are defined over $F$, the identity $\rho(\hat \varphi(x))=\varphi(x)$ implies that
\beq
\rho(\hat \varphi(x)^{-1} \cdot {}^{\sigma^{-1}} \hat \varphi({}^\sigma x))=1.
\eeq
Thus the argument of $\rho$ above gives a Zariski-continuous map $G(\ol F) \to \ker \rho$.  Since $G$ is connected and $\ker \rho$ is discrete, it must be that $\hat \varphi(x)= {}^{\sigma^{-1}} \hat \varphi({}^\sigma x)$, and the lemma follows. \end{proof}

 Therefore:  The $F$-representation $\varphi$ is spinorial iff its extension to ${\ol F}$-points is spinorial.  

 \bibliographystyle{alpha}
\bibliography{refs}

\begin{thebibliography}{Sam90}

\bibitem[Ada69]{Adams}
J.~F. Adams.
\newblock {\em Lectures on {L}ie groups}.
\newblock W. A. Benjamin, Inc., New York-Amsterdam, 1969.

\bibitem[Bou02]{Bou.Lie.4-6}
N.~Bourbaki.
\newblock {\em Lie groups and {L}ie algebras. {C}hapters 4--6}.
\newblock Elements of Mathematics (Berlin). Springer-Verlag, Berlin, 2002.
\newblock Translated from the 1968 French original by Andrew Pressley.

\bibitem[Bou05]{Bou.Lie.7-9}
N.~Bourbaki.
\newblock {\em Lie groups and {L}ie algebras. {C}hapters 7--9}.
\newblock Elements of Mathematics (Berlin). Springer-Verlag, Berlin, 2005.
\newblock Translated from the 1975 and 1982 French originals by Andrew
  Pressley.

\bibitem[Dyn52]{Dynkin.Russian}
E.~B. Dynkin.
\newblock Semisimple subalgebras of semisimple {L}ie algebras.
\newblock {\em Mat. Sbornik N.S.}, 30(72):349--462 (3 plates), 1952.

\bibitem[Dyn00]{Dynkin.Selected}
E.~B. Dynkin.
\newblock {\em Selected papers of {E}. {B}. {D}ynkin with commentary}.
\newblock American Mathematical Society, Providence, RI; International Press,
  Cambridge, MA, 2000.
\newblock Edited by A. A. Yushkevich, G. M. Seitz and A. L. Onishchik.

\bibitem[FH91]{Fulton.Harris}
W.~Fulton and J.~Harris.
\newblock {\em Representation theory}, volume 129 of {\em Graduate Texts in
  Mathematics}.
\newblock Springer-Verlag, New York, 1991.
\newblock A first course, Readings in Mathematics.

\bibitem[GW09]{Goodman.Wallach}
R.~Goodman and N.R. Wallach.
\newblock {\em Symmetry, representations, and invariants}, volume 255 of {\em
  Graduate Texts in Mathematics}.
\newblock Springer, Dordrecht, 2009.

\bibitem[Jan03]{Jantzen}
J.C. Jantzen.
\newblock {\em Representations of algebraic groups}, volume 107 of {\em
  Mathematical Surveys and Monographs}.
\newblock American Mathematical Society, Providence, RI, second edition, 2003.

\bibitem[Jos18]{Rohit.thesis}
R.~Joshi.
\newblock {\em Spinorial Representations of Lie Groups}.
\newblock PhD thesis, Indian Institute of Science Education and Research, Pune,
  2018.

\bibitem[Jr.09]{KRV}
A.~Kirillov Jr.
\newblock {\em An introduction to {L}ie groups and {L}ie algebras}, volume 113
  of {\em Cambridge Studies in Advanced Mathematics}.
\newblock Cambridge University Press, Cambridge, 2009.

\bibitem[Kac90]{Kac}
V.G. Kac.
\newblock {\em Infinite-dimensional {L}ie algebras}.
\newblock Cambridge University Press, Cambridge, third edition, 1990.

\bibitem[Kos76]{Kostant}
B.~Kostant.
\newblock On {M}acdonald's {$\eta $}-function formula, the {L}aplacian and
  generalized exponents.
\newblock {\em Advances in Math.}, 20(2):179--212, 1976.

\bibitem[PR95]{Prasad.Ramakrishnan}
D.~Prasad and D.~Ramakrishnan.
\newblock Lifting orthogonal representations to spin groups and local root
  numbers.
\newblock {\em Proc. Indian Acad. Sci. Math. Sci.}, 105(3):259--267, 1995.

\bibitem[Sam90]{Samelson}
H.~Samelson.
\newblock {\em Notes on Lie Algebras (Universitext)}.
\newblock Springer, 2nd edition, 1 1990.

\bibitem[Spa66]{Spanier}
E.~H. Spanier.
\newblock {\em Algebraic topology}.
\newblock McGraw-Hill Book Co., New York-Toronto, Ont.-London, 1966.

\bibitem[Spr79]{Springer.Corvallis}
T.~A. Springer.
\newblock Reductive groups.
\newblock In {\em Automorphic forms, representations and {$L$}-functions
  ({P}roc. {S}ympos. {P}ure {M}ath., {O}regon {S}tate {U}niv., {C}orvallis,
  {O}re., 1977), {P}art 1}, Proc. Sympos. Pure Math., XXXIII, pages 3--27.
  Amer. Math. Soc., Providence, R.I., 1979.

\bibitem[Spr98]{Springer}
T.~A. Springer.
\newblock {\em Linear algebraic groups}, volume~9 of {\em Progress in
  Mathematics}.
\newblock Birkh\"auser Boston, Inc., Boston, MA, second edition, 1998.

\bibitem[Sta12]{ec1}
R.P. Stanley.
\newblock {\em Enumerative Combinatorics}, volume~1.
\newblock Cambridge University Press, second edition, 2012.

\bibitem[SV00]{Veldkamp}
T.~A. Springer and F.D. Veldkamp.
\newblock {\em Octonions, {J}ordan algebras and exceptional groups}.
\newblock Springer Monographs in Mathematics. Springer-Verlag, Berlin, 2000.

\end{thebibliography}

 \end{document}